\documentclass[reqno]{amsart}
\usepackage{amsmath, amsthm, amssymb, amstext}

\usepackage[left=3.5cm,right=3.5cm,top=3cm,bottom=3cm]{geometry}
\usepackage{hyperref,xcolor}
\hypersetup{
 pdfborder={0 0 0},
 colorlinks,
}
\usepackage{enumitem}
\allowdisplaybreaks

\newtheorem{theorem}{Theorem}
\newtheorem{remark}[theorem]{Remark}
\newtheorem{lemma}[theorem]{Lemma}
\newtheorem{proposition}[theorem]{Proposition}
\newtheorem{corollary}[theorem]{Corollary}
\newtheorem{definition}[theorem]{Definition}

\newcommand{\wto}{ \ \stackrel{w} {\longrightarrow} \ }
\DeclareMathOperator*{\Ss}{S}

\newcommand*\diff{\mathrm{d}}

\DeclareMathOperator*{\graph}{Gr}              %
\DeclareMathOperator*{\divergenz}{div}              %

\newcommand{\N}{\mathbb{N}}

\newcommand{\R}{\mathbb{R}}

\newcommand{\WH}{W^{1, \mathcal{H}}(\Omega)}

\newcommand{\Lp}[1]{L^{#1}(\Omega)}

\newcommand{\Lprand}[1]{L^{#1}(\partial\Omega)}
\newcommand{\Wp}[1]{W^{1,#1}(\Omega)}

\newcommand{\into}{\int_{\Omega}}

\newcommand{\Linf}{L^{\infty}(\Omega)}

\renewcommand{\l}{\left}
\renewcommand{\r}{\right}
\numberwithin{theorem}{section}
\numberwithin{equation}{section}

\title[Double phase implicit obstacle problems]{Double phase implicit obstacle problems with  convection  and multivalued mixed boundary value conditions}

\author[S. Zeng]{Shengda Zeng}
\address[S. Zeng]{Guangxi Colleges and Universities Key Laboratory of Complex System Optimization and Big Data Processing, Yulin Normal University, Yulin 537000, Guangxi, P.R. China, and Jagiellonian University in Krakow,
Faculty of Mathematics and Computer Science, ul. Lojasiewicza 6, 30-348 Krakow, Poland}
\email{zengshengda@163.com}

\author[V.D. R\u{a}dulescu]{Vicen\c{t}iu D. R\u{a}dulescu}
\address[V.D. R\u{a}dulescu]{Faculty of Applied Mathematics, AGH University of Science and Technology, 30-059 Krak\'ow, Poland \& Department of Mathematics, University of Craiova, 200585 Craiova, Romania}
\email[Corresponding author]{radulescu@inf.ucv.ro}

\author[P. Winkert]{Patrick Winkert}
\address[P. Winkert]{Technische Universit\"{a}t Berlin, Institut f\"{u}r Mathematik, Stra\ss e des 17.\,Juni 136, 10623 Berlin, Germany}
\email{winkert@math.tu-berlin.de}


\begin{document}

\begin{abstract}
In this paper we consider a mixed boundary value problem with a nonhomogeneous, nonlinear differential operator (called double phase operator), a nonlinear convection term (a reaction term depending on the gradient), three multivalued terms and an implicit obstacle constraint. Under very general assumptions on the data, we prove that the solution set of such implicit obstacle problem is nonempty (so there is at least one solution) and weakly compact. The proof of our main result uses the Kakutani-Ky Fan fixed point theorem for multivalued operators along with the theory of nonsmooth analysis and variational methods for pseudomonotone operators.
\end{abstract}
	
\maketitle
	
\section{Introduction}
Obstacle problems go back to the pioneering work by Stefan \cite{Stefan} who studied the temperature distribution in a homogeneous medium undergoing a phase change, typically a body of ice at zero degrees centigrade submerged in water. Obstacle problems are also commonly used in physics, biology, and financial mathematics.
Some relevant examples include the dam problem, the Hele-Shaw flow, pricing of American options, quadrature domains, random matrices, etc.

This paper is devoted to the study of a nonlinear partial differential system with a nonlinear convection term, three multivalued terms and an implicit obstacle effect. To this end,  let $\Omega$ be  a bounded domain in $\R^N$, $N\ge 2$, such that its boundary $\Gamma:=\partial \Omega$ is Lipschitz continuous and it is divided into three mutually disjoint parts $\Gamma_1$, $\Gamma_2$ and $\Gamma_3$ with $\Gamma_1$ having positive Lebesgue measure.  It should be pointed out that in our setting the parts $\Gamma_2$ and $\Gamma_3$ can be empty, i.e., $\Gamma_1$ could be the whole boundary $\Gamma_1=\Gamma$. Moreover,  let $1<p<q<N$ and let $\mu\colon\overline \Omega\to[0,\infty)$ be a given function, $U_1\colon \Omega\times \R\to 2^\R$, $U_2\colon \Gamma_2\times \R\to 2^\R$ be two multivalued mappings, $\phi\colon \Gamma_3\times \R\to \R$ be a convex function with respect to the second argument and $f\colon \Omega\times \R\times \R^N\to \R$ be a nonlinear convection function. We study the following problem
\begin{equation}\label{eqn1}
	\begin{aligned}
		-\divergenz\big(|\nabla u|^{p-2}\nabla u+\mu(x)|\nabla u|^{q-2}\nabla u\big)&\quad&&\\
		+|u|^{p-2}u+\mu(x)|u|^{q-2}u&\in U_1(x,u)+f(x,u,\nabla u)  \quad && \text{in } \Omega,\\
		u & = 0 &&\text{on } \Gamma_1,\\
		\frac{\partial u}{\partial \nu_a}&\in U_2(x,u) &&\text{on } \Gamma_2,\\
		-\frac{\partial u}{\partial \nu_a}&\in \partial_c\phi(x,u) &&\text{on } \Gamma_3,\\
		L(u) &\leq J(u),
	\end{aligned}
\end{equation}
where
\begin{align*}
	\frac{\partial u}{\partial \nu_a}:=\left(|\nabla u|^{p-2}\nabla u+\mu(x) |\nabla u|^{q-2}\nabla u\right) \cdot \nu,
\end{align*}
with $\nu$ being the unit normal vector on $\Gamma$, $\partial_c\phi(x,u)$ is the convex subdifferential of $s\mapsto \phi(x,s)$, and $L,J\colon W^{1,\mathcal H}(\Omega)\to \R $ are given functions defined on the Musielak-Orlicz Sobolev space $\WH$, see Section \ref{Section2} for its precise definition.

We point out that problem \eqref{eqn1} combines several interesting phenomena. First
the differential operator involved is the so-called double phase operator which is given by
\begin{align}\label{operator_double_phase}
	\divergenz \big(|\nabla u|^{p-2} \nabla u+ \mu(x) |\nabla u|^{q-2} \nabla u\big)\quad \text{for }u\in \WH.
\end{align}
The corresponding energy functional related to \eqref{operator_double_phase} is given by
\begin{align}\label{integral_minimizer}
	\omega \mapsto \int_\Omega \big(|\nabla  \omega|^p+\mu(x)|\nabla  \omega|^q\big)\,\diff x,
\end{align}
and was initially introduced by Zhikov \cite{Zhikov-1986} in 1986. Such functional was used to describe models for strongly anisotropic materials and it also demonstrated its importance in the study of duality theory as well as in the context
of the Lavrentiev phenomenon, see Zhikov \cite{Zhikov-1995}. The main feature of the variational integral \eqref{integral_minimizer} is the fact that the energy density changes its
ellipticity and growth properties according to the point in the domain. Indeed, the energy density of \eqref{integral_minimizer} exhibits ellipticity in the gradient of order $q$ on the points $x$ where $\mu(x)$ is positive and of order $p$ on the points $x$ where $\mu(x)$ vanishes. In general, double phase differential operators and corresponding energy functionals interpret various comprehensive natural phenomena, and model several problems in Mechanics, Physics and Engineering
Sciences. For instance, in the elasticity theory, the modulating coefficient $\mu(\cdot)$ dictates the geometry of composites made of two different materials with distinct power hardening exponents $p$ and $q$, see Zhikov \cite{Zhikov-2011}.
Further, we mention some famous results in the regularity theory of local minimizers of \eqref{integral_minimizer}, see, for example, the papers of Baroni-Colombo-Mingione \cite{Baroni-Colombo-Mingione-2015,Baroni-Colombo-Mingione-2018}, Colombo-Mingione \cite{Colombo-Mingione-2015a,Colombo-Mingione-2015b} and  Marcellini \cite{Marcellini-1991,Marcellini-1989b}.

A second interesting phenomenon is the appearance of a nonlinearity that depends on the gradient of the solution. Such terms are said to be convection functions. The difficulty in the study of such terms is their nonvariational character, that is, the standard variational tools to corresponding energy functionals cannot be applied. In the past years several interesting works have been published with convection terms, we refer to the papers of de Araujo-Faria \cite{de-Araujo-Faria-2019}, El Manouni-Marino-Winkert \cite{El-Manouni-Marino-Winkert-2022},  Faraci-Motreanu-Puglisi \cite{Faraci-Motreanu-Puglisi-2015}, Faraci-Puglisi \cite{Faraci-Puglisi-2016}, Figueiredo-Madeira \cite{Figueiredo-Madeira-2021}, Gasi\'nski-Papageorgiou \cite{Gasinski-Papageorgiou-2017}, Liu-Motreanu-Zeng \cite{Liu-Motreanu-Zeng-2019}, Marano-Winkert \cite{Marano-Winkert-2019}, Papageorgiou-R\u{a}\-du\-les\-cu-Repov\v{s} \cite{Papageorgiou-Radulescu-Repovs-2020}, and Pucci-Temperini \cite{pucci}.

Further interesting phenomena are the combination of an implicit obstacle effect along with mixed boundary conditions in a very general setting and the appearance of multivalued mappings (which include as special case Clarke's generalized gradients) and convex subdifferentials. We point out that in several critical situations arising in
engineering and economic models, such as Nash equilibrium problems with shared
constraints, semipermeability problems with free boundary conditions, and transport
route optimization with feedback control, the constraint conditions, usually,
depend explicitly on the unknown solution. On the other hand, the theory of (variational-)hemivariational inequalities was originally developed by Panagiotopoulos \cite{Panagiotopoulos-1985,Panagiotopoulos-1993} (see also Naniewicz-Panagiotopoulos \cite{Naniewicz-Panagiotopoulos-1995}) in order
to study the nonsmooth mechanical problems. The main feature in hemivariational
inequalities is to remove the hypotheses on differentiability and convexity
of energy functionals with the help of generalized subgradients introduced by Clarke \cite{Clarke-1983}. Nowadays such problems with implicit obstacle effect and generalized gradient for homogeneous Dirichlet problems have been considered by many authors, see, for example the works of Alleche-R\u{a}dulescu \cite{Alleche-Radulescu-2016}, Aussel-Sultana-Vetrivel \cite{Aussel-Sultana-Vetrivel-2016}, Bonanno-Motreanu-Winkert \cite{Bonanno-Motreanu-Winkert-2011}, Carl-Winkert \cite{Carl-Winkert-2009}, Mig\'{o}rski-Khan-Zeng \cite{Migorski-Khan-Zeng-2020,Migorski-Khan-Zeng-2019}, Gwinner \cite{Gwinner-2018}, Zeng-Mig\'{o}rski-Khan \cite{Zeng-Migorski-Khan-2021} and the references therein. We also refer to the recent monograph of Carl-Le \cite{Carl-Le-2021} about multivalued variational inequalities and inclusions.

Let us comment on some relevant special cases of problem \eqref{eqn1}.

\begin{enumerate}
	\item[\textnormal{(i)}]
	Let $j_1\colon \Omega\times\R\to \R$ and $j_2\colon \Gamma_2\times \R\to \R$ be such that $x\mapsto j(x,s)$ and $x\mapsto j_2(x,s)$ are measurable in $\Omega$ and on $\Gamma_2$, respectively, and $s\mapsto j_1(x,s)$ and $s\mapsto j_2(x,s)$ are both locally Lipschitz continuous. Moreover, let $r_1\colon \R\to \R$ and $r_2\colon \R\to \R$ be two functions and denote by  $\partial j_i$  Clarke's generalized gradient of $j_i$ (with respect to the second variable) for $i=1,2$. If $U_1$ and $U_2$ are defined by $U_1(x,s)=r_1(s)\partial j_1(x,s)$ for a.\,a.\,$x\in \Omega$, $s\in\R$ and $U_2(x,s)=r_2(s)\partial j_2(x,s)$ for a.\,a.\,$x\in\Gamma_2$, $s\in\R$, then problem \eqref{eqn1} becomes
	\begin{equation}\label{eqns2}
		\begin{aligned}
			-\divergenz\big(|\nabla u|^{p-2}\nabla u+\mu(x)|\nabla u|^{q-2}\nabla u\big)&\quad&&\\
			+|u|^{p-2}u+\mu(x)|u|^{q-2}u\in& r_1(u)\partial j_1(x,u)+f(x,u,\nabla u)  \quad && \text{in } \Omega,\\
			u  = &0 &&\text{on } \Gamma_1,\\
			\frac{\partial u}{\partial \nu_a}\in& r_2(u)\partial j_2(x,u) &&\text{on } \Gamma_2,\\
			-\frac{\partial u}{\partial \nu_a}\in& \partial_c\phi(x,u) &&\text{on } \Gamma_3.\\
			L(u) \leq& J(u).
		\end{aligned}
	\end{equation}
	As far as we know problem \eqref{eqns2} has not been studied yet. As a special of our main theorem we state in Theorem \ref{theorem-special-case} that the solution set of \eqref{eqns2} is nonempty and weakly compact.
	\item[\textnormal{(ii)}]
	If $\Gamma_2=\emptyset$ and $\Gamma_3=\emptyset$, i.e., $\Gamma_1=\Gamma$,  then problem \eqref{eqn1} reduces to the following  double phase implicit obstacle system with Dirichlet boundary condition
	\begin{equation}\label{eqns3}
		\begin{aligned}
			-\divergenz\big(|\nabla u|^{p-2}\nabla u+\mu(x)|\nabla u|^{q-2}\nabla u\big)&\quad&\\
			+|u|^{p-2}u+\mu(x)|u|^{q-2}u&\in U_1(x,u)+f(x,u,\nabla u)  \quad && \text{in } \Omega,\\
			u  &= 0 &&\text{on } \Gamma,\\
			L(u) &\leq J(u).
		\end{aligned}
	\end{equation}
	If $U_1(x,s) =-\partial j(x,s)$ with $j\colon \Omega\times \R\to \R $ being locally Lipschitz continuous with respect to the second argument, problem \eqref{eqns3} has been recently studied by Zeng-Bai-Papageorgiou-R\v adulescu~\cite{Zeng-Bai-Papageorgiou-Radulescu-2021}. Further special case when $f$ is independent of $u$ or $U_1\equiv 0$ and $J(u)=\infty$ has been treated by Zeng-Bai-Gasi\'{n}ski-Winkert \cite{Zeng-Bai-Gasinski-Winkert-2020,Zeng-Gasinski-Winkert-Bai-2021,Zeng-Bai-Gasinski-Winkert-2021}. Additionally to $U_1\equiv 0$, if we suppose $J(u)\equiv+\infty$, \eqref{eqns3} reduces to single valued equation studied by Gasi\'{n}ski-Winkert \cite{Gasinski-Winkert-2020b}. Note that the techniques used in these papers differ from ours.
	\item[\textnormal{(iii)}]
	Let $\Psi\colon \Omega\to \R$ be a given obstacle. When $J(u)\equiv0$  and $L(u): =\int_\Omega(u(x)-\Psi(x))^+\,\diff x$ for all $u\in W^{1,\mathcal H}(\Omega)$, then our problem \eqref{eqn1} can be rewritten to the following obstacle inclusion problem
	\begin{equation}\label{eqns4}
		\begin{aligned}
			-\divergenz\big(|\nabla u|^{p-2}\nabla u+\mu(x)|\nabla u|^{q-2}\nabla u\big)&\quad&&\\
			+|u|^{p-2}u+\mu(x)|u|^{q-2}u&\in U_1(x,u)+f(x,u,\nabla u)  \quad && \text{in } \Omega,\\
			u  &= 0 &&\text{on } \Gamma_1,\\
			\frac{\partial u}{\partial \nu_a}&\in U_2(x,u) &&\text{on } \Gamma_2,\\
			-\frac{\partial u}{\partial \nu_a}&\in \partial_c\phi(x,u) &&\text{on } \Gamma_3,\\
			u(x) &\leq \Psi(x)  &&\text{in } \Omega.
		\end{aligned}
	\end{equation}
	To the best of our knowledge, this special case of \eqref{eqn1} has not been considered yet. The same can be said when $\Phi\colon \Gamma\to \R$ is a given obstacle on the boundary. Then the last inequality in \eqref{eqns4} is replaced by $u(x)\leq \Phi(x)$ on $\Gamma$.
	\item[\textnormal{(iv)}]
	If  $J(u)\equiv+\infty$  for all $u\in W^{1,\mathcal H}(\Omega)$, then problem \eqref{eqn1} turns into the following mixed boundary value problem without obstacle effect
	\begin{equation}\label{eqns5}
		\begin{aligned}
			-\divergenz\big(|\nabla u|^{p-2}\nabla u+\mu(x)|\nabla u|^{q-2}\nabla u\big)&\quad&&\\
			+|u|^{p-2}u+\mu(x)|u|^{q-2}u&\in U_1(x,u)+f(x,u,\nabla u)  \quad && \text{in } \Omega,\\
			u  &= 0 &&\text{on } \Gamma_1,\\
			\frac{\partial u}{\partial \nu_a}&\in U_2(x,u) &&\text{on } \Gamma_2,\\
			-\frac{\partial u}{\partial \nu_a}&\in \partial_c\phi(x,u) &&\text{on } \Gamma_3,
		\end{aligned}
	\end{equation}
	which has not been investigated yet. As a corollary we can prove that there exists a solution of \eqref{eqns5} and the solution set of \eqref{eqns5} is weakly compact (see Corollary \ref{corollary} below).
\end{enumerate}

The main goal of the paper is to develop a general framework for determining the existence of a (weak) solution to the nonlinear double phase differential inclusion problem \eqref{eqn1} via employing the Kakutani-Ky Fan fixed point theorem for multivalued operators, the theory of nonsmooth analysis and variational methods for pseudomonotone operators. In fact, to the best of our knowledge, this is the first work which combines a double phase phenomena along with an implicit obstacle constraint,  a nonlinear convection term (i.e., a reaction term depending on the gradient) and multivalued mixed boundary conditions which include a convex subdifferential operator and an abstract multivalued function.

The features of this paper are the following:

(i) the presence of a nonhomogeneous differential operator with different isotropic growth, which generates a double phase
associated energy;

(ii) the analysis developed in this paper is concerned with the  combined effects of a nonstandard operator with unbalanced  growth, a  convection
nonlinearity, three multivalued terms, and an implicit
obstacle constraint;

(iii) the proofs rely on fixed point methods for
multivalued operators in combination with tools from nonsmooth analysis and theory of pseudomonotone operators.

To the best of our knowledge, this is the first paper dealing with multivalued double phase obstacle implicit problems with convection.

In order to give a complete picture of existence results for double phase problems, we also refer to some recent works dealing with different types of problems. In Bahrouni-R\u{a}dulescu-Repov\v{s} \cite{Bahrouni-Radulescu-Repovs-2019} and Bahrouni-R\u{a}dulescu-Winkert \cite{Bahrouni-Radulescu-Winkert-2020} the authors studied a class of double phase energy functionals arising in the theory of transonic flow which are driven by the Baouendi-Grushin operator with variable exponents. Very recently, Biagi-Esposito-Vecchi \cite{Biagi-Esposito-Vecchi-2021} considered positive singular solutions of double phase type equations and proved symmetry as well as monotonicity properties of such solutions. A new class of double phase problems with variable growth has been developed in Cencelj-R\u{a}dulescu-Repov\v{s} \cite{Cencelj-Radulescu-Repovs-2018} while Colasuonno-Squassina \cite{Colasuonno-Squassina-2016} studied existence and properties of related variational eigenvalues. A new type of so-called Finsler double phase operators has been introduced in Farkas-Winkert \cite{Farkas-Winkert-2021} in order to prove an existence result for corresponding singular problems. Furthermore, a Nehari manifold treatment for finding sign changing solutions can be found in Gasi\'nski-Papa\-georgiou \cite{Gasinski-Papageorgiou-2019} for locally Lipschitz right-hand side, in Gasi\'nski-Winkert \cite{Gasinski-Winkert-2021} for nonlinear boundary conditions and in Liu-Dai \cite{Liu-Dai-2018} for superlinear right-hand sides. Existence and multiplicity results for double phase problems with Robin boundary conditions has been recently studied in Papageorgiou-R\u{a}dulescu-Repov\v{s} \cite{Papageorgiou-Radulescu-Repovs-2020b} and
Papageorgiou-Vetro-Vetro \cite{Papageorgiou-Vetro-Vetro-2020}. Perera-Squassina \cite{Perera-Squassina-2019} used a Morse theoretical approach in order to prove the existence of a solution where they used a cohomological local splitting to get an estimate of the critical groups at zero. Existence results for general quasilinear elliptic equations of double phase type in $\R^N$ have been obtained in Zhang-R\u{a}dulescu \cite{Zhang-Radulescu-2018} by using the tools of critical points theory in generalized Orlicz–Sobolev spaces with variable exponent. Finally, we mention the overview articles of R\u{a}dulescu \cite{rna,Radulescu-2019} about isotropic and anistropic double-phase problems and of Mingione-R\u{a}dulescu \cite{Mingione-Radulescu-2021}  about recent developments for problems with nonstandard growth and nonuniform ellipticity.

The paper is organized as follows. Section \ref{Section2} presents a detailed overview about Musielak-Orlicz Sobolev spaces  including several useful embedding results. Moreover, we state some fundamental results in nonsmooth analysis as well as the Kakutani-Ky Fan fixed point theorem for multivalued functions. In Section \ref{Section3}, we first impose the full assumptions on the data of problem \eqref{eqn1} and introduce an auxiliary problem defined in \eqref{eqn3.1}. Then, we apply an existence theorem for a class of mixed variational inequalities involving pseudomonotone operators in which the constraint set is a bounded, closed and convex set, to show that the solution map of the auxiliary problem  \eqref{eqn3.1} is well-defined and completely continuous, see Theorem \ref{Theorems1}. Moreover, we use the Kakutani-Ky Fan fixed point theorem along with the theory of nonsmooth analysis to explore the nonemptiness and compactness of the solution set of problem \eqref{eqn1}. Finally, several special cases of our problem \eqref{eqn1} are discussed and the corresponding existence results are obtained at the end of the paper.

\section{Preliminaries}\label{Section2}

In this section we are going to provide the main definitions and tools which will be needed in the sequel. To this end, let $\Omega\subset \R^N$ be a bounded domain with Lipschitz boundary $\Gamma:=\partial \Omega$ such that $\Gamma$ is divided into three mutually disjoint parts $\Gamma_1$, $\Gamma_2$ and $\Gamma_3$ with $\Gamma_1$ having positive Lebesgue measure. For any fixed $r\in [1,\infty)$ and for any subset $D$ of $\overline \Omega$ we denote the usual Lebesgue spaces by $L^r(D):=L^r(D;\R)$ and $L^r(D;\R^N)$ equipped with the norm $\|\cdot\|_{r,D}$ given by
\begin{align*}
	\|u\|_{r,D}:=\left(\int_D |u|^r\,\diff x\right)^\frac{1}{r}\quad \text{for all }u\in L^r(D).
\end{align*}
We set $L^r(D)_+:=\{u\in L^r(D)\,:\,u(x)\ge 0\text{ for a.a.\,}x\in D\}$ and define by $W^{1,r}(\Omega)$ the corresponding Sobolev space endowed with the norm $\|\cdot\|_{1,r,\Omega}$ given by
\begin{align*}
	\|u\|_{1,r,\Omega}:=\|u\|_{r,\Omega}+\|\nabla u\|_{r,\Omega}\quad \text{for all }u\in W^{1,r}(\Omega).
\end{align*}
Throughout the paper the symbols "$\wto$" and "$\to$" stand for the weak and the strong convergence, respectively. Furthermore, $r'>1$ is the conjugate of $r>1$, that is, $\frac{1}{r}+\frac{1}{r'}=1$. Moreover, we denote by $r^*$ and $r_*$  the critical exponents to $r$ in the domain and on the boundary, respectively, given by
\begin{align}\label{defp1*}
	r^*=
	\begin{cases}
		\frac{Nr}{N-r} &\text{ if }r<N,\\
		+\infty &\text{ if }r\ge N,
	\end{cases}
	\quad\text{and}\quad
	r_*=
	\begin{cases}
		\frac{(N-1)r}{N-r} &\text{ if }r<N,\\
		+\infty &\text{ if }r\ge N,
	\end{cases}
\end{align}
respectively.

Since $\Gamma_1$  has positive measure we know from Korn's inequality that the function space
\begin{align*}
	W:=\l\{u\in W^{1,p}(\Omega)\,:\,u=0\text{ for a.\,a.\,}x\in \Gamma_1\r\}
\end{align*}
equipped with the norm $\|\nabla \cdot\|_{p,\Omega}$, is a reflexive Banach space.  In what follows, let $\lambda>0$ and $\rho>0$ be the smallest constants such that
\begin{align}\label{lamrho}
	\|u\|_{p,\Omega}^p\le \lambda \|\nabla u\|_{p,\Omega}^p
	\quad\text{and}\quad
	\|u\|_{p,\Gamma_2}^p\le \rho\|\nabla u\|_{p,\Omega}^p
\end{align}
for all $u\in W$. From Simon \cite[formula (2.2)]{Simon-1978} we have the well-known inequality
\begin{align}\label{lemmaineqp}
	\left(|x|^{r-2}x-|y|^{r-2}y\r) \cdot \l(x-y\right) \ge k(r)|x-y|^r
\end{align}
for $r\geq 2$ and for all $x,y\in \R^N$, where $k(r)$ is a positive constant.

In the entire paper, we assume that
\begin{align}\label{conditions-p-q-mu}
	1<p<q<N,\quad q<p^*\quad \text{and}\quad 0\leq \mu(\cdot)\in \Linf.
\end{align}
Next we introduce the function $\mathcal H\colon \Omega\times[0,\infty)\to [0,\infty)$ defined  by
\begin{align*}
	\mathcal H(x,t)=t^p+\mu(x)t^q \quad \text{for all }(x,t)\in \Omega\times [0,\infty),
\end{align*}
and recall that $M(\Omega)$ stands for the space of all measurable functions $u\colon\Omega\to\R$, where we identify two such functions which differ on a Lebesgue-null set. Then,  the Musielak-Orlicz Lebesgue space $L^\mathcal{H}(\Omega)$ is given by
\begin{align*}
	L^\mathcal{H}(\Omega)=\left \{u\in M(\Omega)\,:\,\rho_{\mathcal{H}}(u)<+\infty \right \}
\end{align*}
equipped with the Luxemburg norm
\begin{align*}
	\|u\|_{\mathcal{H}} = \inf \left \{ \tau >0\,:\, \rho_{\mathcal{H}}\left(\frac{u}{\tau}\right) \leq 1  \right \},
\end{align*}
where the modular function is given by
\begin{align*}
	\rho_{\mathcal{H}}(u):=\into \mathcal{H}(x,|u|)\,\diff x=\into \big(|u|^{p}+\mu(x)|u|^q\big)\,\diff x.
\end{align*}
We know that $L^\mathcal H(\Omega)$ is uniformly convex, so it is a reflexive Banach space. Additionally, let us introduce the seminormed function space $L_\mu^q(\Omega)$
\begin{align*}
	L^q_\mu(\Omega)=\left \{u\in M(\Omega)\,:\,\into \mu(x) |u|^q\, \diff x< +\infty \right \}
\end{align*}
endowed with the seminorm
\begin{align*}
	\|u\|_{q,\mu} = \left(\into \mu(x) |u|^q \,\diff x \right)^{\frac{1}{q}}.
\end{align*}

Furthermore, we recall that the Musielak-Orlicz Sobolev space $W^{1,\mathcal{H}}(\Omega)$ is given by
\begin{align*}
	W^{1,\mathcal{H}}(\Omega)= \Big \{u \in L^\mathcal{H}(\Omega) \,:\, |\nabla u| \in L^{\mathcal{H}}(\Omega) \Big\}
\end{align*}
equipped with the norm
\begin{align*}
	\|u\|_{1,\mathcal{H}}= \|\nabla u \|_{\mathcal{H}}+\|u\|_{\mathcal{H}},
\end{align*}
where $\|\nabla u\|_\mathcal{H}=\|\,|\nabla u|\,\|_{\mathcal{H}}$. Clearly, $\WH$ is a reflexive Banach space.

Let us now introduce a closed subspace $V$ of $W^{1,\mathcal{H}}(\Omega)$ defined by
\begin{align*}
	V:=\{u\in W^{1,\mathcal{H}}(\Omega)\,:\,u=0\text{ on $\Gamma_1$}\}.
\end{align*}
It is clear that $V$ is also a reflexive Banach space.  In the following we denote the norm of $V$ by $\|u\|_V=  \|u\|_{1,\mathcal{H}}$ for all $u\in V$, and by $\|\cdot\|_{V^*}$ we denote the norm of the dual space $V^*$  of $V$.

The following proposition states the main embeddings for the spaces $\Lp{\mathcal{H}}$ and $\Wp{\mathcal{H}}$, see Gasi\'nski-Winkert \cite[Proposition 2.2]{Gasinski-Winkert-2021} or Crespo-Blanco-Gasi\'nski-Harjulehto-Winkert \cite[Propositions 2.17 and 2.19]{Crespo-Blanco-Gasinski-Harjulehto-Winkert-2021}.

\begin{proposition}\label{proposition_embeddings}
	Let \eqref{conditions-p-q-mu} be satisfied and let $p^*$ as well as $p_*$ be the critical exponents to $p$ as given in \eqref{defp1*} for $r=p$. Then the following embeddings hold:
	\begin{enumerate}
		\item[\textnormal{(i)}]
		$\Lp{\mathcal{H}} \hookrightarrow \Lp{r}$ and $\WH\hookrightarrow \Wp{r}$ are continuous for all $r\in [1,p]$;
		\item[\textnormal{(ii)}]
		$\WH \hookrightarrow \Lp{r}$ is continuous for all $r \in [1,p^*]$ and compact for all $r \in [1,p^*)$;
		\item[\textnormal{(iii)}]
		$\WH \hookrightarrow \Lprand{r}$ is continuous for all $r \in [1,p_*]$ and compact for all $r \in [1,p_*)$;
		\item[\textnormal{(iv)}]
		$\Lp{\mathcal{H}} \hookrightarrow L^q_\mu(\Omega)$ is continuous;
		\item[\textnormal{(v)}]
		$\Lp{q}\hookrightarrow\Lp{\mathcal{H}} $ is continuous.
	\end{enumerate}
\end{proposition}

\begin{remark}
	Note that if we replace the space $W^{1,\mathcal H}(\Omega)$ by $V$ in Proposition \ref{proposition_embeddings}, then the embeddings \textnormal{(ii)} and \textnormal{(iii)} remain valid.
\end{remark}

The following proposition is due to Liu-Dai \cite[Proposition 2.1]{Liu-Dai-2018}.

\begin{proposition}\label{inesp}
	Let \eqref{conditions-p-q-mu} be satisfied and let $y\in \Lp{\mathcal{H}}$. Then the following hold:
	\begin{enumerate}
		\item[{\rm(i)}]
		if $y\neq 0$, then $\|y\|_\mathcal H=\lambda$ if and only if $\rho_\mathcal H\left(\frac{y}{\lambda}\right)=1$;
		\item[{\rm(ii)}]
		$\|y\|_\mathcal H<1$ (resp. $>1$ and $=1$) if and only if $\rho_\mathcal H(y)<1$ (resp. $>1$ and $=1$);
		\item[{\rm(iii)}]
		if $\|y\|_\mathcal H<1$, then $\|y\|_{\mathcal H}^q\le \rho_\mathcal H(y)\le \|y\|_\mathcal H^p$;
		\item[{\rm(iv)}]
		if $\|y\|_\mathcal H>1$, then $\|y\|_{\mathcal H}^p\le \rho_\mathcal H(y)\le \|y\|_\mathcal H^q$;
		\item[{\rm(v)}]
		$\|y\|_\mathcal H\to 0$ if and only if $\rho_\mathcal H(y)\to 0$;
		\item[{\rm(vi)}]
		$\|y\|_\mathcal H\to +\infty$ if and only if $ \rho_\mathcal H(y)\to +\infty$.
	\end{enumerate}
\end{proposition}

Further, we introduce the nonlinear operator $A\colon V\to V^*$ given by
\begin{align}\label{defA}
	\langle A(u),v\rangle :=&\into \big(|\nabla u|^{p-2}\nabla u+\mu(x)|\nabla u|^{q-2}\nabla u \big)\cdot\nabla v \,\diff x\\
	&+\int_\Omega\left(|u|^{p-2}u+\mu(x)|u|^{q-2}u\right)v\,\diff x,\nonumber
\end{align}
for $u,v\in V$ with $\langle \cdot,\cdot\rangle$ being the duality pairing
between $V$ and its dual space $V^*$. The following proposition states the main properties of $A\colon V\to V^*$. We refer to Liu-Dai \cite[Proposition 3.1]{Liu-Dai-2018} or Crespo-Blanco-Gasi\'nski-Harjulehto-Winkert \cite[Proposition 3.4]{Crespo-Blanco-Gasinski-Harjulehto-Winkert-2021} for its proof.

\begin{proposition}\label{prop1}
	The operator $A$ defined by \eqref{defA} is bounded, continuous, monotone (hence maximal monotone) and of type $(\Ss_+)$, that is,
	\begin{align*}
		u_n\wto u \quad \text{in }V \quad\text{and}\quad  \limsup_{n\to\infty}\langle Au_n,u_n-u\rangle\le 0,
	\end{align*}
	imply $u_n\to u$ in $V$.
\end{proposition}

Subsequently, we recall some notion and results concerning nonsmooth analysis and multivalued analysis. We start by recalling definitions and properties of semicontinuous multivalued operators.

\begin{definition}
	Let $Y$ and $Z$ be topological spaces, let $D\subset Y$ be a nonempty set, and let $G\colon Y\to 2^Z$ be a multivalued map.
	\begin{enumerate}
		\item[{\rm(i)}]
		The map $G$ is called upper semicontinuous (u.s.c. for short) at $y\in Y,$ if for each open set $O\subset Z$ such that $G(y)\subset O$,  there exists a neighborhood $N(y)$ of $y$ satisfying $G(N(y)):=\cup_{z\in N(y)}G(z)\subset O$. If it holds for each $y\in D$, then $G$ is called to be upper semicontinuous in $D$.
		\item[{\rm(ii)}]
		The map $G$ is closed at $y\in Y$, if for every sequence $\{(y_n,z_n)\}_{n\in\N}\subset \graph(G)$ satisfying  $(y_n,z_n)\to(y,z)$ in $Y\times Z$, then it holds $(y,z)\in \graph (G)$, where $\graph(G)$ is the graph of $G$ defined by
		\begin{align*}
			\graph(G):=\l\{(y,z)\in Y\times Z\,:\,z\in G(y)\r\}.
		\end{align*}
		If it holds for each $y\in Y$, then $G$ is called to be closed or $G$ has a closed graph.
	\end{enumerate}
\end{definition}

The following proposition provides two useful ways to determinate that a multi-valued map is u.s.c.

\begin{proposition}\label{uscproposition}
	Let $F\colon X\to 2^Y$ with $X$ and $Y$ being topological spaces. The following statements are equivalent:
	\begin{enumerate}
		\item[{\rm(i)}]
		$F$ is upper semicontinuous.
		\item[{\rm(ii)}]
		For each closed set $C\subset Y$, $F^{-}(C):=\{x\in X\,:\, F(x)\cap C\neq\emptyset\}$ is closed in $X$.
		\item[{\rm(iii)}]
		For each open set $O\subset Y$, $F^{+}(O):=\{x\in X\,:\, F(x)\subset O\}$ is open in $X$.
	\end{enumerate}
\end{proposition}

Let  $(E,\|\cdot\|_E)$ be a real Banach space.  Recall that a function $\varphi\colon E\to \overline \R:=\R\cup\{+\infty\}$ is called to be proper, convex and lower semicontinuous, if the following conditions are satisfied:
\begin{enumerate}
	\item[$\bullet$]
	$D(\varphi):=\{u\in E\,:\,\varphi(u)<+\infty\}\neq\emptyset$;
	\item[$\bullet$]
	for any $u,v\in E$ and $t\in(0,1)$, it holds
	$\varphi(tu+(1-t)v)\le t\varphi(u)+(1-t)\varphi(v);$
	\item[$\bullet$]
	$\liminf_{n\to \infty}\varphi(u_n)\ge \varphi(u)$ where the sequence $\{u_n\}_{n\in\N}\subset E$ is such that $u_n\wto u$ as $n\to\infty$ for some $u\in E$.
\end{enumerate}
If the map $\varphi $ is convex, an element $x^* \in E^*$ is called a subgradient of $\varphi$ at $u \in E,$ if
\begin{align}\label{B.SUBCONV}
	\langle x^*, v - u \rangle \le \varphi(v) - \varphi(u) \quad \text{for all} \ v \in E.
\end{align}
The set of all elements $x^* \in E^*$ which satisfies (\ref{B.SUBCONV})
is called the convex subdifferential of $\varphi$ at $u$ and is denoted by $\partial_c \varphi(u)$.

We call a function $j\colon E\to \R$ to be locally Lipschitz at $x\in E$, if there is a neighborhood $O(x)$ of $x$ and a constant $L_x>0$ such that
\begin{align*}
	|j(y)-j(z)|\leq L_x\|y-z\|_E \quad \text{for all } y, z\in O(x).
\end{align*}
Moreover, we denote by
\begin{align*}
	j^\circ(x;y): = \limsup \limits_{z\to x,\, \lambda\downarrow 0 } \frac{j(z+\lambda y)-j(z)}{\lambda},
\end{align*}
the generalized directional derivative of $j$ at the point $x$ in the direction $y$
and $\partial j\colon E\to 2^{E^*}$ given by
\begin{align*}
	\partial j(x): =\left \{\, \xi\in E^{*} \, : \, j^\circ (x; y)\geq \langle\xi, y\rangle_{E^*\times E} \ \text{ for all } y \in    E  \right \} \quad \text{for all } x\in E
\end{align*}
is the generalized gradient of $j$ at $x$ in the sense of Clarke.

The next proposition summarizes the properties of generalized gradients and generalized directional derivatives of a locally Lipschitz functions, see, for example, Mig\'{o}rski-Ochal-Sofonea \cite[Proposition 3.23]{Migorski-Ochal-Sofonea-2013}.

\begin{proposition}\label{P1}
	Let $j\colon E \to \R$ be locally Lipschitz with Lipschitz constant $L_{x}>0$ at $x\in E$. Then we have the following:
	\begin{enumerate}
		\item[{\rm(i)}]
		The function $y\mapsto j^\circ(x;y)$ is positively    homogeneous, subadditive, and satisfies
		\begin{align*}
			|j^\circ(x;y)|\leq L_{x}\|y\|_E \quad \text{for all }y\in E.
		\end{align*}
		\item[{\rm(ii)}]
		The function $(x,y)\mapsto j^\circ(x;y)$ is upper semicontinuous.
		\item[{\rm(iii)}]
		For each $x\in E$, $\partial j(x)$ is a nonempty, convex, and weakly$^*$ compact subset of $E^*$ with $ \|\xi\|_{E^{*}}\leq L_{x}$ for all $\xi\in\partial j(x)$.
		\item[{\rm(iv)}]
		$j^\circ(x;y) = \max\left\{\langle\xi,y\rangle_{E^*\times E} \,:\, \xi\in\partial j(x)\right\}$ for all $y\in E$.
		\item[{\rm(v)}]
		The multivalued function $E\ni x\mapsto \partial j(x)\subset E^*$ is upper semicontinuous from $E$ into w$^*$-$E^*$.
	\end{enumerate}
\end{proposition}

The following theorem provides a useful criterion to determinate the existence of solutions for a class of mixed variational inequalities involving pseudomonotone operators, see, for example, Liu-Liu-Wen-Yao-Zeng \cite[Proposition 5]{Liu-Liu-Wen-Yao-Zeng-2020}.

\begin{theorem}\label{Brodersurjectivity}
	Let $V$ be a reflexive Banach space, $T\colon V\to 2^{V^*}$ be a pseudomonotone operator in the sense of Br\'ezis, $C\subset V$ be nonempty, bounded, closed and convex, and $\varphi\colon V\to \overline\R:=\R\cup\{+\infty\}$ be a proper, convex and l.s.c. function. Then, for a given element $f\in V^*$, there exist $u\in C$ and $u^*\in T(u)$ such that
	\begin{align*}
		\langle u^*-f,v-u\rangle+\varphi(v)-\varphi(u)\ge 0\text{ for all $v\in C$}.
	\end{align*}
\end{theorem}

We end this section by recalling the so-called Kakutani-Ky Fan fixed point theorem,
see, for example, Papageorgiou-Kyritsi-Yiallourou \cite[Theorem 2.6.7]{Papageorgiou-Kyritsi-Yiallourou-2009}. We are going to apply this theorem in order to prove the existence of solutions of problem \eqref{eqn1}.

\begin{theorem}\label{fpt}
	Let $Y$ be a reflexive Banach space and
	$D \subseteq Y$ be a nonempty, bounded, closed and convex set. Let $\Lambda \colon D \to 2^D$
	be a multivalued map with nonempty, closed and convex values such that its graph is sequentially closed in $Y_w \times Y_w$ topology.
	Then $\Lambda$ has a fixed point.
\end{theorem}

\section{Main results}\label{Section3}

This section is devoted to the main results of this paper proving the existence of weak solutions to problem \eqref{eqn1}. First we state the hypotheses on the data of problem \eqref{eqn1}.

\begin{enumerate}
	\item[\textnormal{H($f$):}]
	$f\colon \Omega\times \R\times \R^N\to \R$ is a Carath\'eodory function which satisfies the following conditions:
	\begin{enumerate}
		\item[\textnormal{(i)}]
		there exist $a_f,b_f\ge 0$ and a function $\alpha_f\in L^{\frac{q_1}{q_1-1}}(\Omega)_+$ satisfying
		\begin{align*}
			|f(x,s,\xi)|\le a_f|\xi|^{\frac{p(q_1-1)}{q_1}}+b_f|s|^{q_1-1}+\alpha_f(x)
		\end{align*}
		for a.\,a.\,$x\in \Omega$, for all $s\in\R$ and for all $\xi\in\R^N$, where $1<q_1<p^*$ and $p^*$ is the critical exponents to $p$ in the domain given in \eqref{defp1*};
		\item[\textnormal{(ii)}]
		there exist $c_f,d_f\ge 0$, $\theta_1,\theta_2\in[1,p]$  and a function $\beta_f\in L^1(\Omega)_+$ such that
		\begin{align*}
			f(x,s,\xi)s\le c_f|\xi|^{\theta_1}+d_f|s|^{\theta_2}+\beta_f(x)
		\end{align*}
		for a.\,a.\,$x\in \Omega$, for all $s\in \R$ and for all $\xi\in \R^N$;
		\item[\textnormal{(iii)}]
		there exist $e_f,h_f\ge0$ such that
		\begin{align*}
			(f(x,s,\xi)-f(x,t,\xi))(s-t)&\leq e_f|s-t|^p\\
			|f(x,s,\xi_1)-f(x,s,\xi_2)|&\leq h_f|\xi_1-\xi_2|^{p-1}
		\end{align*}
		for a.\,a.\,$x\in \Omega$, for all $s,t\in\R$ and for all $\xi,\xi_1,\xi_2\in \R^N$.
	\end{enumerate}
\end{enumerate}

\begin{enumerate}
	\item[\textnormal{H($U_1$):}]
	$U_1\colon \Omega\times \R \to 2^\R$ satisfies the following conditions:
	\begin{enumerate}
		\item[\textnormal{(i)}]
		$U_1(x,s)$ is a nonempty, bounded, closed and convex set in $\R$ for a.\,a.\,$x\in \Omega$ and all $s\in\R$;
		\item[\textnormal{(ii)}]
		$x\mapsto U_1(x,s)$ is measurable in $\Omega$ for all $s\in\R$;
		\item[\textnormal{(iii)}]
		$s\mapsto U_1(x,s)$ is u.s.c;
		\item[\textnormal{(iv)}]
		there exist $\theta_3\in [1,p]$, $\alpha_{U_1}\in L^{p'}(\Omega)_+$ and $\beta_{U_1}>0$ such that
		\begin{align*}
			|U_1(x,s)|\le \alpha_{U_1}(x)+\beta_{U_1}|s|^{\theta_3-1}
		\end{align*}
		for a.\,a.\,$x\in \Omega$ and for all $s\in\R$.
	\end{enumerate}
\end{enumerate}

\begin{enumerate}
	\item[\textnormal{H($U_2$):}]
	$U_2\colon \Gamma_2\times \R \to 2^\R$ satisfies the following conditions:
	\begin{enumerate}
		\item[\textnormal{(i)}]
		$U_2(x,s)$ is a nonempty, bounded,  closed and convex set in $\R$ for a.\,a.\,$x\in \Gamma_2$ and all $s\in\R$;
		\item[\textnormal{(ii)}]
		$x\mapsto U_2(x,s)$ is measurable on $\Gamma_2$ for all $s\in\R$;
		\item[\textnormal{(iii)}]
		$s\mapsto U_2(x,s)$ is u.s.c;
		\item[\textnormal{(iv)}]
		there exist $\theta_4\in [1,p]$, $\alpha_{U_2}\in L^{p'}(\Gamma_2)_+$ and $\beta_{U_2}>0$ such that
		\begin{align*}
			|U_2(x,s)|\le \alpha_{U_2}(x)+\beta_{U_2}|s|^{\theta_4-1}
		\end{align*}
		for a.\,a.\,$x\in \Gamma_2$ and for all $s\in\R$.
	\end{enumerate}
\end{enumerate}

\begin{enumerate}
	\item[\textnormal{H($\phi$):}]
	$\phi\colon \Gamma_3\times \R\to \R$ satisfies the following conditions:
	\begin{enumerate}
		\item[\textnormal{(i)}]
		$x\mapsto \phi(x,r)$ is measurable on $\Gamma_3$ for all $r\in\R$ such that $x\mapsto \phi(x,0)$ belongs to $L^1(\Gamma_3)$;
		\item[\textnormal{(ii)}]
		for a.\,a.\,$x\in \Gamma_3$, $r\mapsto \phi(x,r)$ is convex and l.s.c.
	\end{enumerate}
\end{enumerate}

\begin{enumerate}
	\item[\textnormal{H($L$):}]
	$L\colon V\to \R$ is positively homogeneous and subadditive such that
	\begin{align}\label{eqna2}
		L(u)\leq \limsup_{n\to\infty} L(u_n)
	\end{align}
	whenever $\{u_n\}_{n\in\N}\subset V$ is such that $u_n\wto u$ in $V$ for some $u\in V$.
\end{enumerate}

\begin{enumerate}
	\item[\textnormal{H($J$):}]
	$J\colon V\to (0,+\infty)$ is weakly continuous, that is, for any sequence $\{u_n\}_{n\in\N}\subset V$ such that $u_n\wto u$ for some $u\in V$, we have
	\begin{align*}
		J(u_n)\to J(u).
	\end{align*}
\end{enumerate}

\begin{enumerate}
	\item[\textnormal{H($0$):}]
	The inequalities
	\begin{align*}
		\max\{e_f, h_f\lambda^{\frac{1}{p}}\}<k(p)
		\quad\text{and}\quad
		\max\{c_f\delta(\theta_1),d_f\lambda \delta(\theta_2)+\beta_{U_1} \delta(\theta_3) +\beta_{U_2}\delta(\theta_4)\} <1
	\end{align*}
	hold, where $k(p)>0$ is given in \eqref{lemmaineqp}, $\lambda,\rho$ are given in \eqref{lamrho} and $\delta\colon[1,p]\to\{1,0\}$ is defined by
	\begin{align*}
		\delta(\theta)=
		\begin{cases}
			1 &\text{if }\theta=p,\\
			0 & \text{otherwise}.
		\end{cases}
	\end{align*}
\end{enumerate}

\begin{remark}\label{remark1}
	The homogeneity and subadditivity of $L$ reveals that $L$ is a convex function. So, it is not difficult to verify that  if $L\colon V\to \R$ is lower semicontinuous, then inequality \eqref{eqna2} holds automatically.
\end{remark}

We consider the  following multivalued map $K\colon V\to 2^{V}$ given by
\begin{align}\label{defK}
	K(u):=\left\{v\in V\,:\, L(v)\le J(u)\right\}
\end{align}
for all $u\in V$.

Next, we give the definition of a weak solution to problem \eqref{eqn1}.

\begin{definition}\label{weaksol}
	A function $u\in V$ is said to be a weak solution of problem \eqref{eqn1}, if  there exist functions $\eta\in L^{p'}(\Omega)$, $\xi\in L^{p'}(\Gamma_2)$ with $\eta(x)\in U_1(x,u(x))$ for a.\,a.\,$x\in \Omega$, $\xi(x)\in U_2(x,u(x))$ for a.\,a.\,$x\in \Gamma_2$ and if
	\begin{align*}
		&\int_\Omega\left(|\nabla u|^{p-2}\nabla u+\mu(x)|\nabla u|^{q-2}\nabla u\r)\cdot \nabla(v-u) \,\diff x+\int_{\Gamma_3}\phi(x,v)\,\diff\Gamma
		-\int_{\Gamma_3}\phi(x,u)\,\diff\Gamma\\
		&+\int_\Omega\left(|u|^{p-2}u+\mu(x)|u|^{q-2}u\right)(v-u)\,\diff x\\
		&\ge \int_\Omega \eta(x)(v-u)\,\diff x+\int_{\Gamma_2} \xi(x)(v-u)\,\diff\Gamma
		+\int_\Omega f(x,u,\nabla u)(v-u)\,\diff x
	\end{align*}
	is satisfied for all $v\in K(u)$, where the multivalued function $K\colon V\to 2^V$ is defined by \eqref{defK}.
\end{definition}

The following lemma states several properties of $K$, which will be used in the sequel.

\begin{lemma}\label{lemma1}
	Let $J\colon V\to (0,+\infty)$ and $L\colon V\to \R$ be two functions such that  \textnormal{H($L$)} and \textnormal{H($J$)} are satisfied. Then, the following statements hold:
	\begin{enumerate}
		\item[{\rm(i)}]
		for each $u\in V$, $K(u)$ is closed and convex in $V$ such that $0\in K(u)$;
		\item[{\rm(ii)}]
		the graph $\graph(K)$ of $K$ is sequentially closed in $V_w\times V_w$, that is, $K$ is sequentially closed from $V$ with the weak topology into the subsets of  $V$ with the weak topology;
		\item[{\rm(iii)}]
		if  $\{u_n\}_{n\in\N}\subset V$ is a sequence such that
		\begin{align*}
			u_n\wto u\quad \text{in }V \quad\text{as}\quad n\to\infty
		\end{align*}
		for some $u\in V$, then for each $v\in K(u)$ there exists a sequence  $\{v_n\}_{n\in\N}\subset V$ such that
		\begin{align*}
			v_n\in K(u_n)\quad \text{and}\quad v_n\to v \quad \text{ in }V \quad\text{as}\quad n\to\infty.
		\end{align*}
	\end{enumerate}
\end{lemma}

\begin{proof}
	\textnormal{(i)}
	Note that $J(u)>0$ for each $u\in V$ and $L$ is positively homogeneous. So we have $0=L(0_V)<J(u)$ for each $u\in V$, thus, $K(u)\neq\emptyset$ for every $u\in V$. The convexity of $L$ (see Remark~\ref{remark1}) deduces directly that $K(u)$ is convex for each $u\in V$. On the other hand, inequality (\ref{eqna2}) implies that $K(u)$  is closed for each $u\in V$.
	
	\textnormal{(ii)}
	Let $\{(u_n,v_n)\}_{n\in\N}\subset V\times V$ be a sequence  and $(u,v)\in V\times V$ such that $v_n\in K(u_n)$ for each $n\in\mathbb N$ and
	\begin{align*}
		(u_n,v_n)\wto (u,v)\quad \text{in } V\times V\quad \text{as}\quad n\to\infty.
	\end{align*}
	Then, for each $n\in\mathbb N$, it holds $v_n\in K(u_n)$, namely, $L(v_n)\le J(u_n)$. Keeping in mind that $J$ is weakly continuous on $V$ (see hypotheses \textnormal{H($J$)}), we obtain
	\begin{align}\label{eqn3.02}
		\lim_{n\to \infty }J(u_n)=J(u).
	\end{align}
	However, hypotheses \textnormal{H($L$)} turn out
	\begin{align}\label{eqn3.03}
		L(v)\le \limsup_{n\to \infty} L(v_n).
	\end{align}
	Combining \eqref{eqn3.02} and \eqref{eqn3.03} yields
	\begin{align*}
		L(v)\le \limsup_{n\to \infty} L(v_n)\le \limsup_{n\to \infty} J(u_n)=J(u).
	\end{align*}
	This means that $v\in K(u)$, that is, $(u,v)\in \graph(K)$. Therefore, we conclude that $\graph(K)$ is sequentially closed in $V_w\times V_w$.
	
	\textnormal{(iii)}
	Let $\{u_n\}_{n\in\N}\subset V$ be a sequence such that
	\begin{align*}
		u_n\wto u\quad \text{in }V \quad\text{as}\quad n\to\infty
	\end{align*}
	for some $u\in V$ and let $v\in K(u)$ be arbitrary. Because of $J(w)>0$ for every $w\in V$, we construct the sequence $\{v_n\}_{n\in\N}\subset V$ given by
	\begin{align*}
		v_n=\frac{J(u_n)}{J(u)}v.
	\end{align*}
	Due to $v\in K(u)$ (i.e., $L(v)\le J(u)$), it follows from  hypotheses \textnormal{H($L$)} that
	\begin{align*}
		L(v_n)=L\left(\frac{J(u_n)}{J(u)}v\right)=\frac{J(u_n)}{J(u)}L(v)\le \frac{J(u_n)}{J(u)}J(u)=J(u_n).
	\end{align*}
	Therefore, $v_n\in K(u_n)$. Moreover, a simply calculation shows
	\begin{align*}
		\|v_n-v\|_V=\left|\frac{J(u_n)}{J(u)}-1\right|\|v\|_V=0.
	\end{align*}
	Hence, we get that $v_n\to v$ in $ V$ as $n\to\infty$. This completes the proof of the lemma.
\end{proof}

Let $X=L^p(\Omega)\times L^p(\Gamma_2)$ and $(\eta,\xi)\in X^*:=L^{p'}(\Omega)\times L^{p'}(\Gamma_2)$, $w\in V$ be arbitrary. We introduce the following nonlinear auxiliary elliptic unilateral obstacle system
\begin{equation}\label{eqn3.1}
	\begin{aligned}
		-\divergenz\big(|\nabla u|^{p-2}\nabla u+\mu(x)|\nabla u|^{q-2}\nabla u\big)&\quad&&\\
		+|u|^{p-2}u+\mu(x)|u|^{q-2}u&= \eta(x)+f(x,u,\nabla u)  \quad && \text{in } \Omega,\\
		u &=  0 &&\text{on } \partial \Gamma_1,\\
		\frac{\partial u(x)}{\partial \nu_a}&=\xi(x)  &&\text{on } \partial \Gamma_2,\\
		-\frac{\partial u(x)}{\partial \nu_a} &\in\partial_c\phi(x,u) &&\text{on } \partial \Gamma_3,\\
		L(u) &\leq J(w).
	\end{aligned}
\end{equation}

By virtue of Definition \ref{weaksol}, a function $u\in V$ is said to be a weak solution of problem \eqref{eqn3.1}, if the following holds: $u\in K(w)$ and
\begin{align*}
	\begin{split}
		&\int_\Omega\left(|\nabla u|^{p-2}\nabla u+\mu(x)|\nabla u|^{q-2}\nabla u\r)\cdot \nabla(v-u)\,\diff x
		+\int_{\Gamma_3}\phi(x,v)\,\diff\Gamma-\int_{\Gamma_3}\phi(x,u)\,\diff\Gamma\\
		&+\int_\Omega\left(|u|^{p-2}u+\mu(x)|u|^{q-2}u\right)(v-u)\,\diff x\\
		&\geq \int_\Omega \eta(x)(v-u)\,\diff x+\int_{\Gamma_2} \xi(x)(v-u)\,\diff\Gamma
		+\int_\Omega f(x,u,\nabla u)(v-u)\,\diff x
	\end{split}
\end{align*}
for all $v\in K(w)$.

Next, we prove the existence and uniqueness of weak solutions to problem \eqref{eqn3.1}.

\begin{theorem}\label{Theorems1}
	Let $p\ge 2$. Suppose that \eqref{conditions-p-q-mu}, \textnormal{H($f$)}, \textnormal{H($L$)} and \textnormal{H($\phi$)} are satisfied. If, in addition, the inequality $\max\{e_f, h_f\lambda^{\frac{1}{p}}\}< k(p)$ holds , then, for every $(w,(\eta,\xi))\in V\times X^*$, there exists a unique  weak  solution $u\in K(w)$ of problem \eqref{eqn3.1}.
\end{theorem}

\begin{proof}
	{\bf Uniqueness:}
	Let $(w,(\eta,\xi))\in V\times X^*$  be fixed and assume that the solution set of  problem \eqref{eqn3.1} associated with $(w,(\eta,\xi))\in V\times X^*$ is nonempty.  Let $u_1,u_2\in V$ be two weak solutions of  problem \eqref{eqn3.1} corresponding to $(w,(\eta,\xi))\in V\times X^*$. This means that $u_i\in K(w)$ and
	\begin{align*}
		&\int_\Omega\left(|\nabla u_i|^{p-2}\nabla u_i+\mu(x)|\nabla u_i|^{q-2}\nabla u_i\r)\cdot \nabla(v-u_i)\,\diff x
		+\int_{\Gamma_3}\phi(x,v)\,\diff\Gamma\\
		&-\int_{\Gamma_3}\phi(x,u_i)\,\diff\Gamma+\int_\Omega\left(|u_i|^{p-2}u_i+\mu(x)|u_i|^{q-2}u_i\right)(v-u_i)\,\diff x\\
		&\ge \int_\Omega \eta(x)(v-u_i)\,\diff x
		+\int_{\Gamma_2} \xi(x)(v-u_i)\,\diff\Gamma
		+\int_\Omega f(x,u,\nabla u_i)(v-u_i)\,\diff x
	\end{align*}
	for all $v\in K(w)$ and for $i=1,2$.  Taking $v=u_2$ and $v=u_1$ for the inequalities above with $i=1$ and $i=2$, respectively, we sum up the resulting inequalities to get
	\begin{align*}
		&\int_\Omega\left(|\nabla u_1|^{p-2}\nabla u_1+\mu(x)|\nabla u_1|^{q-2}\nabla u_1\r)\cdot\nabla(u_1-u_2)\,\diff x\\
		&\quad-\int_\Omega\left(|\nabla u_2|^{p-2}\nabla u_2+\mu(x)|\nabla u_2|^{q-2}\nabla u_2\r)\cdot \nabla(u_1-u_2)\,\diff x\\
		&\quad+\int_\Omega\left(|u_1|^{p-2}u_1+\mu(x)|u_1|^{q-2}u_1\right)(u_1-u_2)\,\diff x\\
		&\quad-\int_\Omega\left(|u_2|^{p-2}u_2+\mu(x)|u_2|^{q-2}u_2\right)(u_1-u_2)\,\diff x\\
		&\le \int_\Omega (f(x,u_1,\nabla u_1)-f(x,u_2,\nabla u_2))(u_1-u_2)\,\diff x\\
		&=\int_\Omega (f(x,u_1,\nabla u_1)-f(x,u_2,\nabla u_1))(u_1-u_2)\,\diff x\\
		&\quad +\int_\Omega (f(x,u_2,\nabla u_1)-f(x,u_2,\nabla u_2))(u_1-u_2)\,\diff x.
	\end{align*}
	Taking \eqref{lemmaineqp}, \eqref{conditions-p-q-mu} and \textnormal{H($f$)(iii)} into account implies
	\begin{align*}
		& k(p)\left(\|\nabla u_1-\nabla u_2\|_{p,\Omega}^p+\| u_1- u_2\|_{p,\Omega}^p\right)\\
		&\leq \int_\Omega e_f|u_1-u_2|^p\,\diff x+\int_\Omega h_f|\nabla u_1-\nabla u_2|^{p-1}|u_1-u_2|\,\diff x.
	\end{align*}
	Applying H\"older's inequality and \eqref{lamrho} gives
	\begin{align*}
		&k(p)\left(\|\nabla u_1-\nabla u_2\|_{p,\Omega}^p+\| u_1- u_2\|_{p,\Omega}^p\right)\\
		&\leq e_f\|u_1-u_2\|_{p,\Omega}^p+ h_f\|\nabla u_1-\nabla u_2\|_{p,\Omega}^{p-1}\|u_1-u_2\|_{p,\Omega}\\
		& \le e_f \| u_1- u_2\|_{p,\Omega}^p+ h_f\lambda^{\frac{1}{p}}\|\nabla u_1-\nabla u_2\|_{p,\Omega}^p.
	\end{align*}
	Hence,
	\begin{align*}
		\min\left\{(k(p)-e_f),(k(p)-h_f\lambda^{\frac{1}{p}})\right\}\left(\|\nabla u_1-\nabla u_2\|_{p,\Omega}^p+\| u_1- u_2\|_{p,\Omega}^p\right)\le0.
	\end{align*}
	By assumption, we know that $\max\{e_f, h_f\lambda^{\frac{1}{p}}\}< k(p)$, thus $u_1=u_2$. Therefore, for each $(w,(\eta,\xi))\in V\times X^*$  problem \eqref{eqn3.1} has a unique weak solution $u\in V$ provided the solution set of problem \eqref{eqn3.1} is nonempty.
	
	{\bf Existence:}
	Let $n\in \mathbb N$ be fixed such that $K_n:=\overline {B(0,n)}\cap K(w)\neq\emptyset$, where $\overline{B(0,n)}$ is the closed ball centered at the origin with radius $n>0$, that is, $\overline{B(0,n)}:=\{u\in V\,:\,\|u\|_V\le n\}$.
	We first consider the following auxiliary problem: find $u_n\in K_n$ such that
	\begin{align}\label{eqns3.03}
		\begin{split}
			&\int_\Omega\left(|\nabla u_n|^{p-2}\nabla u_n+\mu(x)|\nabla u_n|^{q-2}\nabla u_n\r)\cdot \nabla(v-u_n)\,\diff x\\
			&+\int_\Omega\left(|u_n|^{p-2}u_n+\mu(x)|u_n|^{q-2}u_n\right)(v-u_n)\,\diff x\\
			&\quad +\int_{\Gamma_3}\phi(x,v)\,\diff\Gamma -\int_{\Gamma_3}\phi(x,u_n)\,\diff\Gamma\\
			& \geq \int_\Omega \eta(x)(v-u_n)\,\diff x
			+\int_{\Gamma_2} \xi(x)(v-u_n)\,\diff\Gamma\\
			&\quad +\int_\Omega f(x,u_n,\nabla u_n)(v-u_n)\,\diff x\quad\text{for all }v\in K_n.
		\end{split}
	\end{align}
	Let us consider the function $\varphi\colon V\to \R\cup\{+\infty\}$ defined by
	\begin{align}\label{defvarphi}
		\varphi(u):=\int_{\Gamma_3}\phi(x,u)\,\diff\Gamma \quad\text{for all }u\in V.
	\end{align}
	Moreover, let $N_f\colon V\subset L^{q_1}(\Omega)\to L^{q_1}(\Omega)$ be the Nemytskij operator associated to $f$ and let $\iota$ be the embedding operator from $V$ to $L^{q_1}(\Omega)$ with its adjoint operator $\iota^*\colon L^{q_1'}(\Omega)\to V^*$.  We set $T\colon V\to V^*$
	\begin{align*}
		Tu=Au-\iota^*N_f(u)\quad \text{for all }u\in V,
	\end{align*}
	where $A\colon V\to V^*$ is defined by \eqref{defA}. Based on this, it is easy to see that problem (\ref{eqns3.03}) can be rewritten equivalently to the following one: find $u_n\in K_n$ such that
	\begin{align}\label{eqns3.04}
		\langle Tu_n,v-u_n\rangle+\varphi(v)-\varphi(u_n)\ge \langle g,v-u_n\rangle
	\end{align}
	for all $v\in K_n$, where $g\in V^*$ is defined by (because of Riesz's representative theorem)
	\begin{align*}
		\langle g,w\rangle:= \int_\Omega \eta(x)w(x)\,\diff x+\int_{\Gamma_2} \xi(x)w(x)\,\diff\Gamma\quad \text{for all }w\in V.
	\end{align*}
	
	Note that $K_n$ is a bounded, closed and convex subset of $V$.  We are going to apply Theorem \ref{Brodersurjectivity} to prove that problem \eqref{eqns3.04} has at least one solution. By hypothesis \textnormal{H($\phi$)} it is obvious to see that $\varphi$ defined in \eqref{defvarphi} is a proper, convex and l.s.c.\,function. Reasoning as in the proof of Theorem 3.2 in Gasi\'nski-Winkert \cite{Gasinski-Winkert-2020b}, we conclude that $T$ is a bounded and pseudomonotone operator. Therefore, all conditions of Theorem \ref{Brodersurjectivity} are satisfied and so for each $n\in \mathbb N$ there exists one solution $u_n\in K_n$ of problem \eqref{eqns3.04}. Applying the same arguments as in the proof of uniqueness part, we conclude that $u_n$ is the unique solution to problem \eqref{eqns3.04}.
	
	Furthermore, we claim that there exists a number $N_0\in \mathbb N$ such that the unique solution $u_{N_0}$ of problem \eqref{eqns3.04} with $n=N_0$ satisfies the inequality
	\begin{align}\label{eqns3.05}
		\|u_{N_0}\|_V<N_0.
	\end{align}
	Let us suppose \eqref{eqns3.05} is not true. Then for each $n\in \mathbb N$ we have $\|u_n\|_V=n$. So, it holds $\|u_n\|_V\to +\infty$ as $n\to\infty$.
	
	Recall that $0\in K(w)$ (see Lemma \ref{lemma1}(i)), a simple calculation shows
	\begin{equation}\label{eqns3.07}
		\begin{split}
			&\int_\Omega|\nabla u_n|^p+\mu(x)|\nabla u_n|^q+|u_n|^p+\mu(x)|u_n|^{q}\,\diff x+\int_{\Gamma_3}\phi(x,u_n)\,\diff\Gamma\\
			&\le \int_\Omega \eta(x)u_n\,\diff x+\int_{\Gamma_2} \xi(x)u_n\,\diff\Gamma+\int_{\Gamma_3}\phi(x,0)\,\diff\Gamma
			+\int_\Omega f(x,u_n,\nabla u_n)u_n\,\diff x.
		\end{split}
	\end{equation}
	Since $\varphi$ is a proper, convex and l.s.c.\,function, there are constants $\alpha_\varphi>0$ and $\beta_\varphi\ge 0$ satisfying
	\begin{align}\label{eqns3.08}
		\varphi(v)\ge -\alpha_\varphi\|v\|_V-\beta_\varphi\text{ for all $v\in V$},
	\end{align}
	owing to Br\'{e}zis \cite[Proposition 1.10]{Brezis-2011}. Let $\varepsilon_1,\varepsilon_2>0$. Employing hypothesis \textnormal{H($f$)(ii)}, we have
	\begin{align}\label{eqns3.09}
		\begin{split}
			&\int_\Omega f(x,u_n,\nabla u_n)u_n(x)\,\diff x\\
			&\leq \int_\Omega c_f|\nabla u_n(x)|^{\theta_1}+d_f|u_n(x)|^{\theta_2}+\beta_f(x)\,\diff x\\
			&\leq \begin{cases}
				c_f\|\nabla u\|_{p,\Omega}^p+d_f\|u_n\|_{p,\Omega}^p+\|\beta_f\|_{1,\Omega}&\text{if }\theta_1=\theta_2=p,\\
				c_f\|\nabla u_n\|_{p,\Omega}^p+\varepsilon_1\|u_n\|_{p,\Omega}^p+c_1(\varepsilon_1)+\|\beta_f\|_{1,\Omega}&\text{if }\theta_1=p \text{ and }\theta_2<p,\\
				\varepsilon_2\|\nabla u_n\|_{p,\Omega}^p+c_2(\varepsilon_2)+d_f\|u_n\|_{p,\Omega}^p+\|\beta_f\|_{1,\Omega}&\text{if }\theta_1<p \text{ and }\theta_2=p,\\
				\varepsilon_2\|\nabla u_n\|_{p,\Omega}^p+c_2(\varepsilon_2)+\varepsilon_1\|u_n\|_{p,\Omega}^p+c_1(\varepsilon_1)+\|\beta_f\|_{1,\Omega}&\text{if }\theta_1<p \text{ and }\theta_2<p,
			\end{cases}
		\end{split}
	\end{align}
	for some $c_1(\varepsilon_1),c_2(\varepsilon_2)>0$.  Keeping in mind that $(\eta,\xi)\in X^*$, we use H\"older's inequality to get
	\begin{align}\label{eqns3.010}
		\int_\Omega \eta(x)u_n\,\diff x+\int_{\Gamma_2} \xi(x)u_n\,\diff\Gamma
		\le \|\eta\|_{p',\Omega}\|u_n\|_{p,\Omega} +\|\xi\|_{p',\Gamma_2}\|u_n\|_{p,\Gamma_2}.
	\end{align}
	Combining \eqref{eqns3.08}, \eqref{eqns3.09} and \eqref{eqns3.010} with \eqref{eqns3.07} and using \eqref{lamrho}, \eqref{conditions-p-q-mu} as well as Young's inequality we obtain
	\begin{align*}
		\int_\Omega|\nabla u_n|^p+\mu(x)|\nabla u_n|^q+|u_n|^p+\mu(x)|u_n|^{q}\,\diff x\le c_3
	\end{align*}
	for some $c_3>0$. Using Proposition~\ref{inesp}, it is not difficult to see that the inequality above is a contradiction to the fact that $\|u_n\|_V\to +\infty$ as $n\to\infty$. Therefore, there exists a number $N_0\in\mathbb N$ such that (\ref{eqns3.05}) is satisfied.
	
	For any $z\in K(w)$, we can choose $t\in(0,1)$ small enough such that $v_t=u_{N_0}+(1-t)z\in K_{N_0}$ thanks to (\ref{eqns3.05}). Taking $v=v_t$ in \eqref{eqns3.04} gives
	\begin{align*}
		\langle Tu_{N_0},z-u_{N_0}\rangle+\varphi(z)-\varphi(u_{N_0})\ge \langle g,z-u_{N_0}\rangle.
	\end{align*}
	Since $z\in K(w)$ is arbitrary, we see that $u_{N_0}\in K(w)$ is a weak solution of problem \eqref{eqn3.1}.
	
	In conclusion, we have shown that for every $(w,(\eta,\xi))\in V\times X^*$ problem \eqref{eqn3.1} has a unique solution.
\end{proof}

Particularly, when $J(w)=+\infty$ for all $w\in V$, then problem \eqref{eqn3.1} becomes the following non-obstacle nonlinear elliptic system with mixed boundary conditions: find $u\in V$ such that
\begin{equation}\label{eqns3.011}
	\begin{aligned}
		-\divergenz\big(|\nabla u|^{p-2}\nabla u+\mu(x)|\nabla u|^{q-2}\nabla u\big)&\quad&&\\
		+|u|^{p-2}u+\mu(x)|u|^{q-2}u&= \eta(x)+f(x,u,\nabla u)  \quad && \text{in } \Omega,\\
		u&=   0 &&\text{on } \Gamma_1,\\
		\frac{\partial u}{\partial \nu_a}&= \xi(x) &&\text{on } \Gamma_2,\\
		-\frac{\partial u}{\partial \nu_a}&\in \partial_c\phi(x,u) &&\text{on } \Gamma_3.\\
	\end{aligned}
\end{equation}

We have the following corollary for problem \eqref{eqns3.011}

\begin{corollary}
	Let $p\ge 2$. Assume that \eqref{conditions-p-q-mu}, \textnormal{H($f$)} and \textnormal{H($\phi$)}  are fulfilled. If the inequality $\max\{e_f, h_f\lambda^{\frac{1}{p}}\}<k(p)$ holds, then for each pair of function $(\eta,\xi)\in  X^*$,  problem (\ref{eqns3.011}) has a unique weak solution.
\end{corollary}

We define $\gamma \colon V\to X$ by $\gamma u=(\gamma_1u,\gamma_2u)$ for all $u\in V$, where $\gamma_1\colon V\to L^p(\Omega)$ is the embedding operator of $V$ to $L^p(\Omega)$ and $\gamma_2\colon V\to L^p(\Gamma_2)$ is the trace operator. From Proposition \ref{proposition_embeddings} we know that $\gamma$ is a linear, bounded and compact operator. Next we introduce the multivalued mapping $U\colon X\to 2^{X^*}$ defined by
\begin{align}\label{defU}
	U(u,v):=\big\{(\eta,\xi)\in X^*\,:\, \eta(x)\in U_1(x,u(x)) \text{ a.\,a.\,in }\Omega,\ \xi(x)\in U_2(x,v(x))\text{ a.\,a.\,on }\Gamma_2\big\}	
\end{align}
for all $(u,v)\in X$.

\begin{lemma}\label{lemma2}
	Assume that \textnormal{H($U_1$)} and \textnormal{H($U_2$)} are satisfied. Then, the following hold:
	\begin{enumerate}
		\item[{\rm(i)}]
		$U$ is well-defined and for each $(u,v)\in X$ the set $U(u,v)$ is bounded, closed and convex in $X^*$;
		\item[{\rm(ii)}]
		$U$ is strongly-weakly u.s.c.
	\end{enumerate}
\end{lemma}

\begin{proof}
	\textnormal{(i)}
	Note that $U_1$ and $U_2$ satisfy an upper Carath\'eodory condition, that is, $x\mapsto U_1(x,s)$ and $x\mapsto U_2(x,s)$ are measurable and $s\mapsto U_1(x,s)$ and $s\mapsto U_2(x,s)$ are u.s.c. Therefore, we can apply Theorem 1.3.4 of Kamenskii-Obukhovskii-Zecca \cite{Kamenskii-Obukhovskii-Zecca-2001} to conclude that for each $(u,v)\in X$ the functions $x\mapsto U_1(x,u(x))$ and $x\mapsto U_2(x,v(x))$ are both measurable in $\Omega$ and on $\Gamma_2$, respectively. Additionally, from the Yankov-von Neumann-Aumann selection theorem (see e.\,g.\,Papageorgiou-Winkert \cite[Theorem 2.7.25]{Papageorgiou-Winkert-2018}), we are able to find measurable selections $\eta\colon \Omega \to \R$ and $\xi\colon \Gamma_2\to \R$ such that $\eta(x)\in U_1(x,u(x))$ for a.\,a.\,$x\in \Omega$ and $\xi(x)\in U_2(x,v(x))$ for a.\,a.\,$x\in \Gamma_2$. From hypotheses \textnormal{H($U_1$)}(iv) and \textnormal{H($U_2$)}(iv) we have that
	\begin{align}\label{eqns3.012}
		\begin{split}
			\|\eta\|_{p',\Omega}^{p'}
			&=\int_\Omega|\eta(x)|^{p'}\,\diff x\le \int_\Omega\left(\alpha_{U_1}(x)+\beta_{U_1}|u(x)|^{\theta_3-1}\right)^{p'}\,\diff x\\[2ex]
			&\le
			\begin{cases}
				\displaystyle\int_\Omega\left(\alpha_{U_1}(x)+\beta_{U_1}|u(x)|^{p-1}\right)^{p'}\,\diff x&\text{ if $\theta_3=p$},\\[3ex]
				\displaystyle\int_\Omega\left(\alpha_{U_1}(x)+c_4+|u(x)|^{p-1}\right)^{p'}\,\diff x&\text{ if $\theta_3<p$},
			\end{cases}\\[2ex]
			&\le
			\begin{cases}
				\displaystyle\int_\Omega c_5\left(|\alpha_{U_1}(x)|^{p'}+\beta_{U_1}|u(x)|^{p}\right)\,\diff x&\text{ if $\theta_3=p$},\\[3ex]
				\displaystyle\int_\Omega c_5\left(|\alpha_{U_1}(x)|^{p'}+1+|u(x)|^{p}\right)\,\diff x&\text{ if $\theta_3<p$},
			\end{cases}\\[2ex]
			&<+\infty,
		\end{split}
	\end{align}
	for some $c_4,c_5>0$, and
	\begin{align}\label{eqns3.013}
		\begin{split}
			\|\xi\|_{p',\Gamma_2}^{p'}
			&=\int_{\Gamma_2}|\xi(x)|^{p'}\,\diff\Gamma\le \int_{\Gamma_2}\left(\alpha_{U_2}(x)+\beta_{U_2}|v(x)|^{\theta_4-1}\right)^{p'}\,\diff\Gamma\\[2ex]
			&\le
			\begin{cases}
				\displaystyle\int_{\Gamma_2}\left(\alpha_{U_2}(x)+\beta_{U_2}|v(x)|^{p-1}\right)^{p'}\,\diff\Gamma&\text{ if $\theta_4=p$},\\[3ex]
				\displaystyle\int_{\Gamma_2}\left(\alpha_{U_2}(x)+c_6+|v(x)|^{p-1}\right)^{p'}\,\diff\Gamma&\text{ if $\theta_4<p$},
			\end{cases}\\[2ex]
			&\le
			\begin{cases}
				\displaystyle\int_{\Gamma_2} c_7\left(|\alpha_{U_2}(x)|^{p'}+\beta_{U_2}|v(x)|^{p}\right)\,\diff\Gamma&\text{ if $\theta_4=p$},\\[3ex]
				\displaystyle\int_{\Gamma_2} c_7\left(|\alpha_{U_2}(x)|^{p'}+1+|v(x)|^{p}\right)\,\diff\Gamma&\text{ if $\theta_4<p$},
			\end{cases}\\[2ex]
			&<+\infty,
		\end{split}
	\end{align}
	for some $c_6,c_7>0$, where we have used Young inequality for the cases $\theta_3<p$ and $\theta_4<p$. This means that $U$ is well-defined and for each $(u,v)\in X$ the set $U(u,v)$ is bounded in $X^*$. Moreover, the closedness and convexity of $U_1$ and $U_2$ guarantee that the set $U(u,v)$ is closed and convex in $X^*$ for all $(u,v)\in X$.
	
	\textnormal{(ii)}
	From Proposition \ref{uscproposition}, it is sufficient to prove that for each weakly closed set $D$ of $X^*$, the set $U^-(D)$ is closed in $X$. Let $\{(u_n,v_n)\}_{n\in\N}\subset U^-(D)$ be such that $(u_n,v_n)\to (u,v)$ in $X$ as $n\to\infty$ for some $(u,v)\in X$. Without any loss of generality, we may assume that
	\begin{align}\label{eqns3.014}
		\begin{split}
			u_n(x)&\to u(x)\quad\text{as }n\to\infty \quad \text{for a.\,a.\,}x\in \Omega, \\
			v_n(x)&\to v(x)\quad\text{as }n\to\infty \quad \text{for a.\,a.\,}x\in \Gamma_2.
		\end{split}
	\end{align}
	Let $\{(\eta_n,\xi_n)\}_{n\in\N}\subset X^*$ be a sequence such that $(\eta_n,\xi_n)\in U(u_n,v_n)$ for each $n\in \mathbb N$. From \eqref{eqns3.012} and \eqref{eqns3.013}, we can see that $\{(\eta_n,\xi_n)\}_{n\in\N}$ is bounded in $X^*$. The reflexivity of $X$ permits us to assume that, passing to a subsequence if necessary, $(\eta_n,\xi_n)\wto (\eta,\xi)$ in $X^*$ as $n\to \infty$ for some $(\eta,\xi)\in D$, owing to the weak closedness of $D$. Our objective is to prove that $(\eta,\xi)\in U(u,v)$, that is, $\eta(x)\in U_1(x,u(x))$ for a.\,a.\,$x\in \Omega$ and $\xi(x)\in U_2(x,v(x))$ for a.\,a.\,$x\in \Gamma_2$. Subsequently, we only show that $\eta(x)\in U_1(x,u(x))$ for a.\,a.\,$x\in \Omega$, the other proof works similarly.
	
	From Mazur's theorem we know that there exists a sequence $\{\zeta_n\}_{n\in\N}$ of convex combinations of $\{\eta_n\}_{n\in\N}$ such that
	\begin{align}\label{eqns3.015}
		\zeta_n\to \eta\quad \text{in }L^{q'}(\Omega)\quad \text{and}\quad \zeta_n(x)\to u(x)\quad \text{for a.\,a.\,}x\in \Omega\quad \text{as}\quad n\to\infty.
	\end{align}
	The convexity of $U_1$ reveals that $\zeta_n(x)\in U_1(x,u_n(x))$ for a.\,a.\,$x\in \Omega$. Applying the convergences in \eqref{eqns3.014} and \eqref{eqns3.015} along with the upper semicontinuity of $U_1$ (see hypothesis \textnormal{H($U_1$)(iii)}), we get that $\eta(x)\in U_1(x,u(x))$ for a.\,a.\,$x\in \Omega$. Likewise, it is true that $\xi(x)\in U_2(x,v(x))$ for a.\,a.\,$x\in \Gamma_2$. This means that $(\eta,\xi)\in U(u,v)\cap D$, thus, $(u,v)\in U^-(D)$. Therefore, we can apply Proposition~\ref{uscproposition} to conclude that $U$ is strongly-weakly u.s.c.
\end{proof}

By virtue of Theorem~\ref{Theorems1}, we are now in a position to consider  the solution mapping  $S\colon V\times X^*\to V$ of problem \eqref{eqn3.1} defined by
\begin{align*}
	S(w,(\eta,\xi)):=u(w,(\eta,\xi))\quad \text{for all }(w,(\eta,\xi))\in V\times X^*,
\end{align*}
where $u(w,(\eta,\xi))\in V$ is the unique solution of problem \eqref{eqn3.1} corresponding to $(w,(\eta,\xi))\in V\times X^*$.

The following lemma indicates that the solution map $S$ is completely continuous, that is, for any sequence $\{(w_n,(\eta_n,\xi_n))\}_{n\in\N}\subset V\times X^*$ and $(u,(\eta,\xi))\in V\times X^*$ with $(w_n,(\eta_n,\xi_n))\wto (w,(\eta,\xi))$ in $V\times X^*$ as $n\to\infty$ it holds $S(w_n,(\eta_n,\xi_n))\wto S(w,(\eta,\xi))$ in $V$ as $n\to\infty$.

\begin{lemma}\label{lemma3}
	Let $p\ge 2$. Assume that \eqref{conditions-p-q-mu}, \textnormal{H($f$)}, \textnormal{H($\phi$)},  \textnormal{H($L$)}, \textnormal{H($J$)} and \textnormal{H($0$)} are fulfilled. Then, the solution map $S\colon V\times X^*\to V$ of problem \eqref{eqn3.1} is completely continuous.
\end{lemma}

\begin{proof}
	Let $\{(w_n,(\eta_n,\xi_n))\}_{n\in\N}\subset V\times X^*$, $\{u_n\}_{n\in\N}\subset V$  be sequences and $(w,(\eta,$ $\xi))\in V\times X^*$ such that
	\begin{align*}
		(w_n,(\eta_n,\xi_n))\wto (w,(\eta,\xi))\text{ in $V\times X^*$ as $n\to\infty$}
	\end{align*}
	and  $u_n=S(w_n,(\eta_n,\xi_n))$ for each $n\in \mathbb N$. Then, for each $n\in\mathbb N$, we have $u_n\in K(w_n)$ and
	\begin{equation}\label{eqn3.9}
		\begin{split}
			&\int_\Omega\left(|\nabla u_n|^{p-2}\nabla u_n+\mu(x)|\nabla u_n|^{q-2}\nabla u_n\r)\cdot \nabla(v-u_n)\,\diff x\\
			&\quad+\int_\Omega\left(|u_n|^{p-2}u_n+\mu(x)|u_n|^{q-2}u_n\right)(v-u_n)\,\diff x\\
			&\quad +\int_{\Gamma_3}\phi(x,v)\,\diff\Gamma-\int_{\Gamma_3}\phi(x,u_n)\,\diff\Gamma\\
			&\ge \int_\Omega \eta_n(x)(v-u_n)\,\diff x
			+\int_{\Gamma_2} \xi_n(x)(v-u_n)\,\diff\Gamma\\
			&\quad+\int_\Omega f(x,u_n,\nabla u_n)(v-u_n)\,\diff x \quad \text{for all }v\in K(w_n).
		\end{split}
	\end{equation}
	
	{\bf Claim:} $\{u_n\}_{n\in\N}$ is bounded in $V$.
	
	Arguing by contradiction, without any loss of generality, we may assume that $\|u_n\|_V\to +\infty$ as $n\to\infty$. Lemma~\ref{lemma1}(i) indicates that $0\in K(w_n)$ for each $n\in \mathbb N$.  Taking $v=0$ in \eqref{eqn3.9} yields
	\begin{align*}
		&\int_\Omega|\nabla u_n|^{p}+\mu(x)|\nabla u_n|^{q}+|u_n|^p+\mu(x)|u_n|^q\,\diff x+\int_{\Gamma_3}\phi(x,u_n(x))\,\diff\Gamma\\
		&\le \int_\Omega \eta_n(x)u_n(x)\,\diff x+\int_{\Gamma_2} \xi_n(x)u_n(x)\,\diff\Gamma+\int_\Omega f(x,u_n,\nabla u_n)u_n(x)\,\diff x\\
		&\quad
		+\int_{\Gamma_3}\phi(x,0)\,\diff\Gamma\\
		&\le  \|\phi(\cdot,0)\|_{1,\Gamma_3} +\int_\Omega c_f|\nabla u_n(x)|^{\theta_1}+d_f|u_n(x)|^{\theta_2}+\beta_f(x)\,\diff x+\|\eta_n\|_{p',\Omega}\|u_n\|_{p,\Omega}\\
		&\quad +\|\xi_n\|_{p',\Gamma_2}\|u_n\|_{p,\Gamma_2},
	\end{align*}
	where we have used H\"older's inequality and hypotheses \textnormal{H($f$)(ii)} and \textnormal{H($\phi$)(i)}. Using   \eqref{lamrho}, \eqref{conditions-p-q-mu}, \eqref{eqns3.08}, \eqref{eqns3.09} and inequality $\max\{c_f\delta(\theta_1),d_f \delta(\theta_2)\}<1$ we obtain
	\begin{align*}
		\int_\Omega|\nabla u_n|^{p}+\mu(x)|\nabla u_n|^{q}+|u_n|^p+\mu(x)|u_n|^q\,\diff x\le c_8
	\end{align*}
	for some $c_8>0$ which is independent of $n$ thanks to the boundedness of $\{(\eta_n,\xi_n)\}_{n\in\N}$ in $X^*$. Combining the inequality above and Proposition~\ref{inesp}, we get a contradiction. So, the Claim is proved.
	
	From the Claim we may assume, passing to a subsequence if necessary, that
	\begin{align*}
		u_n \wto u \quad \text{in }V\quad  \text{as}\quad n\to\infty
	\end{align*}
	for some $u\in V$.
	
	We assert that $u=S(w,(\eta,\xi))$, i.e., $u$ is the unique solution of problem \eqref{eqn3.1} corresponding to $(w,(\eta,\xi))\in V\times X^*$. Lemma~\ref{lemma1}(ii) indicates that $\graph(K)$ is sequentially closed in $V_w\times V_w$. Keeping in mind $\{(u_n,w_n)\}_{n\in\N}\subset \graph(K)$ and $(u_n,w_n)\wto (u,w)$ in $V\times V$ as $n\to\infty$, we have $u\in K(w)$. On the other hand, it follows from Lemma~\ref{lemma1}(iii) that there exists a sequence $\{y_n\}_{n\in\N}\subset V$ satisfying $y_n\in K(w_n)$ and
	\begin{align*}
		y_n\to u\quad \text{in }V \quad \text{as}\quad n\to\infty.
	\end{align*}
	Inserting $v=y_n$ in \eqref{eqn3.9} and passing to the upper limit as $n\to\infty$ to the resulting inequality, we have
	\begin{align*}
		&\limsup_{n\to\infty}\int_\Omega\left(|\nabla u_n|^{p-2}\nabla u_n+\mu(x)|\nabla u_n|^{q-2}\nabla u_n\r) \cdot \nabla(u_n-y_n)\,\diff x\\
		&\quad +\limsup_{n\to\infty}\int_\Omega\left(|u_n|^{p-2}u_n+\mu(x)|u_n|^{q-2}u_n\right)(u_n-y_n)\,\diff x\\
		&\le \limsup_{n\to\infty}\int_\Omega f(x,u_n,\nabla u_n)(u_n-y_n)\,\diff x+
		\limsup_{n\to\infty}\int_{\Gamma_3}\phi(x,y_n)\,\diff\Gamma  \\
		& \quad -\liminf_{n\to\infty}\int_{\Gamma_3}\phi(x,u_n)\,\diff\Gamma +\limsup_{n\to\infty}\int_{\Gamma_2} \xi_n(x)(y_n-u_n)\,\diff\Gamma\\
		&\quad		+ \limsup_{n\to\infty}\int_\Omega \eta_n(x)(u_n-y_n)\,\diff x.
	\end{align*}
	Because the embeddings of $V$ into $L^{q_1}(\Omega)$, $V$ into $L^{p}(\Omega)$ and $V$ into $L^{p}(\Gamma_2)$ are all compact (due to $q_1<p^*$, see Proposition \ref{proposition_embeddings}), the sequences $\{(\eta_n,\xi_n)\}_{n\in\N}$ and $\{f(\cdot,u_n,\nabla u_n)\}_{n\in\N} $ are bounded in $X^*$ and $L^{q_1'}(\Omega)$ (see hypothesis \textnormal{H($f$)}(i)), respectively. We obtain
	\begin{align*}
		& \limsup_{n\to\infty}\int_\Omega \eta_n(x)(u_n-y_n)\,\diff x+\limsup_{n\to\infty}\int_{\Gamma_2} \xi_n(x)(u_n-y_n)\,\diff\Gamma\\
		&+\limsup_{n\to\infty}\int_\Omega f(x,u_n,\nabla u_n)(u_n-y_n)\,\diff x=0.
	\end{align*}
	From the convexity and continuity of $\phi$ (hence, $V\ni u\mapsto \int_{\Gamma_3}\phi(x,u(x))\,\diff\Gamma$ is weakly l.s.c.), it is not difficult to see that
	\begin{align*}
		\limsup_{n\to\infty}\int_{\Gamma_3}\phi(x,y_n(x))\,\diff\Gamma
		-\liminf_{n\to\infty}\int_{\Gamma_3}\phi(x,u_n(x))\,\diff\Gamma\le 0.
	\end{align*}
	Taking into account the last three inequalities, we have
	\begin{align*}
		&\limsup_{n\to\infty}\int_\Omega \left (|\nabla u_n|^{p-2} \nabla u_n  +\mu(x)|\nabla u_n|^{q-2} \nabla u_n\r)\cdot \nabla (u_n-u) \,\diff x\\
		&+\limsup_{n\to\infty}\int_\Omega\left(|u_n|^{p-2}u_n+\mu(x)|u_n|^{q-2}u_n\right)(u_n-u)\,\diff x \\
		&\le \limsup_{n\to\infty}\int_\Omega \left (|\nabla u_n|^{p-2} \nabla u_n  +\mu(x)|\nabla u_n|^{q-2} \nabla u_n\r)\cdot \nabla (u_n-u)\,\diff x\\
		& +\liminf_{n\to\infty}\int_\Omega \left (|\nabla u_n|^{p-2} \nabla u_n  +\mu(x)|\nabla u_n|^{q-2} \nabla u_n\r)\cdot \nabla (u-y_n) \,\diff x\\
		&+\limsup_{n\to\infty}\int_\Omega\left(|u_n|^{p-2}u_n+\mu(x)|u_n|^{q-2}u_n\right)(u_n-u)\,\diff x \\
		&+\liminf_{n\to\infty}\int_\Omega\left(|u_n|^{p-2}u_n+\mu(x)|u_n|^{q-2}u_n\right)(u-y_n)\,\diff x \\
		&\le \limsup_{n\to\infty}\int_\Omega \left (|\nabla u_n|^{p-2} \nabla u_n  +\mu(x)|\nabla u_n|^{q-2} \nabla u_n\r)\cdot \nabla (u_n-y_n)\,\diff x\\
		&+\limsup_{n\to\infty}\int_\Omega\left(|u_n|^{p-2}u_n+\mu(x)|u_n|^{q-2}u_n\right)(u_n-y_n)\,\diff x \\
		&\le 0.
	\end{align*}
	The latter together with the convergence $u_n\wto u$ in $V$ as $n\to\infty$,  and Proposition~\ref{prop1} (i.e., $A$ is of type $(\Ss_+)$) deduces that
	\begin{align*}
		u_n\to u\quad \text{in }V \quad \text{as} \quad n\to\infty.
	\end{align*}
	
	For any fixed $v\in K(w)$, Lemma~\ref{lemma1}(iii) permits us to find a  sequence $\{v_n\}_{n\in\N}\subset V$ such that
	\begin{align*}
		v_n\in K(w_n)\quad \text{and}\quad  v_n\to v \quad \text{in }V\quad \text{as}\quad n\to\infty.
	\end{align*}
	Taking $v=v_n$  in \eqref{eqn3.9} and passing to the upper limit as $n\to\infty$ we  obtain
	\begin{align*}
		&\int_\Omega\left(|\nabla u|^{p-2}\nabla u+\mu(x)|\nabla u|^{q-2}\nabla u\r)\cdot \nabla(v-u)\,\diff x+\int_{\Gamma_3}\phi(x,v)\,\diff\Gamma
		-\int_{\Gamma_3}\phi(x,u)\,\diff\Gamma\\
		&\quad+\int_\Omega\left(|u|^{p-2}u+\mu(x)|u|^{q-2}u\right)(v-u)\,\diff x\\
		&\ge \limsup_{n\to\infty}\int_\Omega\left(|\nabla u_n|^{p-2}\nabla u_n+\mu(x)|\nabla u_n|^{q-2}\nabla u_n\r)\cdot  \nabla(v_n-u_n) \,\diff x\\
		&\quad +\limsup_{n\to\infty}\int_{\Gamma_3}\phi(x,v_n)\,\diff\Gamma-\liminf_{n\to\infty}\int_{\Gamma_3}\phi(x,u_n)\,\diff\Gamma\\
		&\quad+\limsup_{n\to\infty}\int_\Omega\left(|u_n|^{p-2}u_n+\mu(x)|u_n|^{q-2}u_n\right)(v_n-u_n)\,\diff x\\
		&\ge \limsup_{n\to\infty}\int_\Omega \eta_n(x)(v_n-u_n)\,\diff x+\limsup_{n\to\infty}\int_{\Gamma_2} \xi_n(x)(v_n-u_n)\,\diff\Gamma\\
		&\quad +\limsup_{n\to\infty}\int_\Omega f(x,u_n,\nabla u_n)(v_n-u_n)\,\diff x\\
		&\ge \int_\Omega \eta(x)(v-u)\,\diff x+\int_{\Gamma_2} \xi(x)(v-u)\,\diff\Gamma
		+\int_\Omega f(x,u,\nabla u)(v-u)\,\diff x.
	\end{align*}
	Note that $v\in K(w)$ is arbitrary, so we conclude that $u$ is the unique solution of problem \eqref{eqn3.1} corresponding to $(w,(\eta,\xi))$, that is, $u=S(w,(\eta,\xi))$. Since every convergent subsequence of $\{u_n\}_{n\in\N}$ converges to the same limit $u$,  we know that the whole sequence $\{u_n\}_{n\in\N}$ converges strongly to $u$ in  $V$. This means that $S\colon V\times X^*\to V$ is completely continuous.
\end{proof}

The following lemma gives a priori estimates for the weak solutions of problem \eqref{eqn1}.

\begin{lemma}\label{lemmaes}
	Let $p\ge 2$. Suppose that \eqref{conditions-p-q-mu}, \textnormal{H($f$)}, \textnormal{H($\phi$)}, \textnormal{H($U_1$)}, \textnormal{H($U_2$)}, \textnormal{H($L$)}, \textnormal{H($J$)} and \textnormal{H($0$)} are satisfied. If  the solution set of problem \eqref{eqn1}, denoted by $\Upsilon$, is nonempty, then there exists a constant $M>0$ such that
	\begin{align}\label{bounedesp}
		\|u\|_{V}\le M\quad \text{for all }u\in \Upsilon.
	\end{align}
\end{lemma}

\begin{proof}
	Let $u\in V$ be a weak solution of problem \eqref{eqn1}. Then, there exist functions $(\eta,\xi)\in X^*$ with $\eta(x)\in  U_1(x,u(x))$ for a.\,a.\,$x\in \Omega$ and $\xi(x)\in U_2(x,u(x))$ for a.\,a.\,$x\in \Gamma_2$ such that
	\begin{align*}
		&\int_\Omega\left(|\nabla u|^{p-2}\nabla u+\mu(x)|\nabla u|^{q-2}\nabla u\r)\cdot \nabla(v-u)\,\diff x
		+\int_{\Gamma_3}\phi(x,v)\,\diff\Gamma
		-\int_{\Gamma_3}\phi(x,u)\,\diff\Gamma\\
		&+\int_\Omega\left(| u|^{p-2} u+\mu(x)| u|^{q-2} u\r)(v-u)\,\diff x\\
		&\geq \int_\Omega f(x,u,\nabla u)(v-u)\,\diff x
		+\int_\Omega \eta(x)(v-u)\,\diff x
		+\int_{\Gamma_2} \xi(x)(v-u)\,\diff\Gamma
	\end{align*}
	for all $v\in K(u)$. Since $0\in K(u)$, we can take $v=0$ in the inequality above to get
	\begin{equation}\label{eqns3.016}
		\begin{split}
			&\int_\Omega\left(|\nabla u|^{p-2}\nabla u+\mu(x)|\nabla u|^{q-2}\nabla u\r)\cdot \nabla u+|u|^p+\mu(x)|u|^q\,\diff x\\
			&\le \int_{\Gamma_3}\phi(x,0)\,\diff\Gamma-\int_{\Gamma_3}\phi(x,u)\,\diff\Gamma
			+\int_\Omega f(x,u,\nabla u)u\,\diff x
			+\int_\Omega \eta(x)u\,\diff x+\int_{\Gamma_2} \xi(x)u\,\diff\Gamma.
		\end{split}
	\end{equation}
	Let $\varepsilon_3,\varepsilon_4>0$. From hypotheses \textnormal{H($U_1$)}(iv) and \textnormal{H($U_2$)}(iv) it follows that
	\begin{align}\label{eqns.3.018b}
		\begin{split}
			\int_\Omega \eta(x)u\,\diff x
			&\le \int_\Omega |\eta(x)||u|\,\diff x
			\le \int_\Omega \l(\alpha_{U_1}(x)+\beta_{U_1}|u|^{\theta_3-1}\r)|u|\,\diff x\\
			&\le
			\begin{cases}
				\displaystyle\int_\Omega \alpha_{U_1}(x)|u|
				+\beta_{U_1}|u|^{p}\,\diff x&\text{if }\theta_3=p,\\[3ex]
				\displaystyle\int_\Omega (\alpha_{U_1}(x)+c_9(\varepsilon_3))|u|+\varepsilon_3|u|^{p}\,\diff x&\text{if } \theta_3<p,
			\end{cases}\\[1ex]
			&\leq
			\begin{cases}
				\displaystyle\|\alpha_{U_1}\|_{p',\Omega}\|u\|_{p,\Omega}
				+\beta_{U_1}\|u\|_{p,\Omega}^{p}&\text{if }\theta_3=p,\\[1ex]
				\|\alpha_{U_1}(\cdot)+c_9(\varepsilon_3)\|_{p',\Omega}\|u\|_{p,\Omega}+\varepsilon_3\|u\|_{p,\Omega}^{p}&\text{if }\theta_3<p,
			\end{cases}
		\end{split}
	\end{align}
	for some $c_9(\varepsilon_3)>0$, and
	\begin{align}\label{eqns3.018}
		\begin{split}
			\int_{\Gamma_2} \xi(x)u\,\diff\Gamma
			&\le \int_{\Gamma_2} |\xi(x)||u|\,\diff\Gamma
			\le \int_{\Gamma_2} (\alpha_{U_2}(x)+\beta_{U_2}|u|^{\theta_4-1})|u|\,\diff\Gamma\\
			&\leq
			\begin{cases}
				\displaystyle\int_{\Gamma_2} \alpha_{U_2}(x)|u|+\beta_{U_2}|u|^{p}\,\diff\Gamma&\text{if } \theta_4=p,\\[3ex]
				\displaystyle\int_{\Gamma_2} (\alpha_{U_2}(x)+c_{10}(\varepsilon_4))|u|+\varepsilon_4|u|^{p}\,\diff\Gamma&\text{if }\theta_4<p,
			\end{cases}\\[1ex]
			&\leq
			\begin{cases}
				\displaystyle\|\alpha_{U_2}\|_{p',\Gamma_2}\|u\|_{p,\Gamma_2}+\beta_{U_2}\|u\|_{p,\Gamma_2}^{p}&\text{if }\theta_4=p,\\
				\|\alpha_{U_2}(\cdot)+c_{10}(\varepsilon_3)\|_{p',\Gamma_2}\|u\|_{p,\Gamma_2}+\varepsilon_4\|u\|_{p,\Gamma_2}^{p}&\text{if }\theta_4<p,
			\end{cases}
		\end{split}
	\end{align}
	for some $c_{10}(\varepsilon_4)>0$, where we have used Young's inequality (for the cases $\theta_3<p$ and $\theta_4<p$) and H\"older's inequality. Taking \eqref{conditions-p-q-mu}, \eqref{eqns3.09}, \eqref{eqns3.010} and \eqref{eqns3.016}, \eqref{eqns.3.018b}, \eqref{eqns3.018} into account, we take $\varepsilon_3$ (for the case $\theta_3<p$) and $\varepsilon_4$ (for the case $\theta_4<p$) small enough and use the inequality $\max\{c_f\delta(\theta_1),d_f \delta(\theta_2)+\beta_{U_1} \delta(\theta_3) +\beta_{U_2} \delta(\theta_4)\}<1$ to infer that there exists a constant $c_{11}>0$ which is independent of $u$ such that
	\begin{align*}
		\int_\Omega\left(|\nabla u|^{p-2}\nabla u+\mu(x)|\nabla u|^{q-2}\nabla u\r)\cdot \nabla u+|u|^p+\mu(x)|u|^q\,\diff x\le c_{11}.
	\end{align*}
	Therefore, it follows that the solution set of problem \eqref{eqn1} is bounded, namely, there exists a constant $M>0$ such that \eqref{bounedesp} is satisfied. This completes the proof of the lemma.
\end{proof}

Now we are in the position to state and prove our main result in this paper concerning the nonemptiness and weak compactness of the solution set $\Upsilon$ to problem \eqref{eqn1}.

\begin{theorem}\label{maintheorem}
	Let $p\ge 2$. Under the assumptions of Lemma~\ref{lemma3},  the solution set $\Upsilon$ of problem \eqref{eqn1} is nonempty and weakly compact in $V$.
\end{theorem}

\begin{proof}
	{\bf Nonemptiness:}
	We consider the multivalued mapping $\Xi\colon V\times X^*\to 2^{V\times X^*}$ defined by
	\begin{align*}
		\Xi(u,(\eta,\xi)):=(S(u,(\eta,\xi)), U(\gamma u)),
	\end{align*}
	where $U$ is given in (\ref{defU}) and $\gamma u=(\gamma_1u,\gamma_2u)$ for all $u\in V$. From Theorem \ref{Theorems1} and Lemma \ref{lemma2}, we know that $\Xi$ has nonempty, bounded, closed and convex values. Observe that if $(u,(\eta,\xi))$ is a fixed point of $\Xi$, then we have $u=S(u,(\eta,\xi))$ and $(\eta,\xi)\in U(\gamma u)$. By the definitions of $S$ and $U$ it is not difficult to see that $u$ is also a weak solution of problem \eqref{eqn1}. Based on this fact, we are going to apply the Kakutani-Ky Fan fixed point theorem stated as Theorem \ref{fpt}, to verify that the fixed point set of $\Xi$ is nonempty.
	
	First, we claim that there exists a constant $M^*>0$ such that
	\begin{align}\label{bounedinc}
		S(\overline{B_V(0,M^*)},\ U(\gamma \overline{B_V(0,M^*)}))\subset \overline{B_V(0,M^*)},
	\end{align}
	where $\overline{B_V(0,M^*)}:=\{u\in V\,:\,\|u\|_V\le M^*\}$. Arguing by contradiction, suppose there is not a constant $M^*$ such that the inclusion holds.  Then for each $n>0$ there exist $w_n,z_n\in \overline{B_V(0,n)}$ and $(\eta_n,\xi_n)\in X^*$ with $(\eta_n,\xi_n)\in U(\gamma z_n)$ such that \begin{align*}
		u_n=S(w_n,(\eta_n,\xi_n)) \quad \text{and}\quad \|u_n\|_V>n.
	\end{align*}
	Hence, for every $n>0$, we have
	\begin{align*}
		&\int_\Omega\left(|\nabla u_n|^{p-2}\nabla u_n+\mu(x)|\nabla u_n|^{q-2}\nabla u_n\r)\cdot \nabla(v-u_n)\,\diff x\\
		&+\int_\Omega\left(| u_n|^{p-2} u_n+\mu(x)| u_n|^{q-2} u_n\r) (v-u_n)\,\diff x
		\\	&+\int_{\Gamma_3}\phi(x,v)\,\diff\Gamma
		-\int_{\Gamma_3}\phi(x,u_n)\,\diff\Gamma\\
		&\ge \int_\Omega \eta_n(x)(v-u_n)\,\diff x
		+\int_{\Gamma_2} \xi_n(x)(v-u_n)\,\diff\Gamma
		+\int_\Omega f(x,u_n,\nabla u_n)(v-u_n)\,\diff x
	\end{align*}
	for all $v\in K(w_n)$. Choosing $v=0$ in the inequality above, we obtain
	\begin{align}\label{eqns3.019}
		\begin{split}
			&\int_\Omega\left(|\nabla u_n|^{p-2}\nabla u_n+\mu(x)|\nabla u_n|^{q-2}\nabla u_n\r)\cdot \nabla u_n+|u_n|^p+\mu(x)|u_n|^q\,\diff x\\
			&\le \int_{\Gamma_3}\phi(x,0)\,\diff\Gamma
			-\int_{\Gamma_3}\phi(x,u_n)\,\diff\Gamma
			+ \int_\Omega \eta_n(x)u_n\,\diff x
			+\int_{\Gamma_2} \xi_n(x)u_n\,\diff\Gamma\\
			&\quad +\int_\Omega f(x,u_n,\nabla u_n)u_n\,\diff x.
		\end{split}
	\end{align}
	Let $\varepsilon_5,\varepsilon_6>0$. Applying hypotheses \textnormal{H($U_1$)}(iv) and \textnormal{H($U_2$)}(iv) gives
	\begin{align} \label{eqns3.020}
		\begin{split}
			\int_\Omega \eta_n(x)u_n\,\diff x
			&\le \int_\Omega (\alpha_{U_1}(x)+\beta_{U_1}|z_n|^{\theta_3-1})|u_n|\,\diff x\\
			&\leq
			\begin{cases}
				\displaystyle\int_\Omega \l(\alpha_{U_1}(x)
				+\beta_{U_1}|z_n|^{p-1}\r)|u_n|\,\diff x&\text{if }\theta_3=p,\\[3ex]
				\displaystyle\int_\Omega \l(\l(\alpha_{U_1}(x)+c_{12}(\varepsilon_5)\r)|u_n|+\varepsilon_5|z_n|^{p-1}|u_n|\r)\,\diff x&\text{if }\theta_3<p,
			\end{cases}\\
			&\le
			\begin{cases}
				\displaystyle\|\alpha_{U_1}\|_{p',\Omega}\|u_n\|_{p,\Omega}
				+\beta_{U_1}\|z_n\|_{p,\Omega}^{p-1}\|u_n\|_{p,\Omega}&\text{if }\theta_3=p,\\
				\|\alpha_{U_1}(\cdot)+c_{12}(\varepsilon_5)\|_{p',\Omega}\|u_n\|_{p,\Omega}+\varepsilon_6\|z_n\|_{p,\Omega}^{p-1}\|u_n\|_{p,\Omega}&\text{if }\theta_3<p,
			\end{cases}
		\end{split}
	\end{align}
	for some $c_{12}(\varepsilon_5)>0$, and
	\begin{align}\label{eqns3.021}
		\begin{split}
			\int_{\Gamma_2} \xi_n(x)u_n\,\diff\Gamma
			&\le \int_{\Gamma_2} \l(\alpha_{U_2}(x)+\beta_{U_2}|z_n|^{\theta_4-1}\r)|u_n|\,\diff\Gamma\\
			&\leq
			\begin{cases}
				\displaystyle\int_{\Gamma_2} \l(\alpha_{U_2}(x)+\beta_{U_2}|z_n|^{p-1}\r)|u_n(x)|\,\diff\Gamma&\text{if } \theta_4=p,\\[3ex]
				\displaystyle\int_{\Gamma_2} \l(\alpha_{U_2}(x)+c_{13}(\varepsilon_6)\r)|u_n|+\varepsilon_6|z_n|^{p-1}|u_n|\,\diff\Gamma&\text{if }\theta_4<p,
			\end{cases}\\
			& \leq
			\begin{cases}
				\displaystyle\|\alpha_{U_2}\|_{p',{\Gamma_2}}\|u_n\|_{p,{\Gamma_2}}+\beta_{U_2}\|z_n\|_{p,{\Gamma_2}}^{p-1}\|u_n\|_{p,{\Gamma_2}}&\text{if } \theta_4=p,\\
				\|\alpha_{U_2}(\cdot)+c_{13}(\varepsilon_6)\|_{p',{\Gamma_2}}\|u_n\|_{p,{\Gamma_2}}+\varepsilon_5\|z_n\|_{p,{\Gamma_2}}^{p-1}\|u_n\|_{p,{\Gamma_2}}&\text{if }\theta_4<p,
			\end{cases}
		\end{split}
	\end{align}
	for some $c_{13}(\varepsilon_6)>0$. Taking  \eqref{eqns3.08}, \eqref{eqns3.09} and \eqref{eqns3.019}, \eqref{eqns3.020}, \eqref{eqns3.021} into account, we are able to find a constant $c_{14}>0$ which is independent of $n$ such that
	\begin{align*}
		\int_\Omega\left(|\nabla u_n|^{p-2}\nabla u_n+\mu(x)|\nabla u_n|^{q-2}\nabla u_n\r)\cdot \nabla u_n+|u_n|^p+\mu(x)|u_n|^q\,\diff x\le c_{14}
	\end{align*}
	for all $n>0$,
	where we have used the fact that $\|z_n\|_V\le n=\|u_n\|_V$.  Letting $n\to\infty$ for the inequality above, it leads to a contradiction. So, there exists a constant $M^*>0$ such that \eqref{bounedinc} is fulfilled.
	
	From hypotheses \textnormal{H($U_1$)(iv)} and \textnormal{H($U_2$)(iv)}, we can see that $U\colon X\to 2^{X^*}$ is a bounded mapping. Let $M_1>0$ be such that
	\begin{align*}
		\|U(\gamma \overline{B_V(0,M^*)})\|_{X^*}\le M_1.
	\end{align*}
	Now we introduce a bounded, closed and convex subset $D$ of $V\times X^*$ given by
	\begin{align*}
		D=\l\{(u,(\eta,\xi))\in V\times X^*\,:\,\|u\|_V\le M^*\text{ and }\|(\eta,\xi)\|_{X^*}\le M_1\r\}.
	\end{align*}
	From this  and \eqref{bounedinc} we see that $\Xi$ maps $D$ into itself.
	
	Next, we are going to prove that the graph of $\Xi$ is sequentially closed in $(V\times X^*)_w\times (V\times X^*)_w$. Let  $\{(w_n,(\eta_n,\xi_n))\}_{n\in\N}\subset V\times X^*$ and $\{(u_n,(\delta_n,\sigma_n))\}_{n\in\N}\subset V\times X^*$  be sequences such that $(\delta_n,\sigma_n)\in U(\gamma (w_n))$, $u_n=S(w_n,(\eta_n,\xi_n))$ and
	\begin{align*}
		(w_n,(\eta_n,\xi_n)&\wto(w,(\eta,\xi))\quad \text{in }V\times X^* \quad \text{as }n\to\infty\\
		(u_n,(\delta_n,\sigma_n))&\wto(u,(\delta,\sigma))\quad \text{in }V\times X^*\quad \text{as }n\to\infty
	\end{align*}
	for some $(w,(\eta,\xi))\in V\times X^*$ and $(u,(\delta,\sigma))\in V\times X^*$. From Lemma \ref{lemma3} we know that
	\begin{align*}
		u_n=S(w_n,(\eta_n,\xi_n))\to S(w,(\eta,\xi))\quad \text{in } V\times X^*\quad \text{as }n\to\infty.
	\end{align*}
	This means that $u=S(w,(\eta,\xi))$. Recall that $U$ is strongly-weakly u.s.c.\,with boun\-ded, closed and convex values (see Lemma \ref{lemma2}), it allows us to apply Theorem 1.1.4 of  Kamenskii-Obukhovskii-Zecca \cite{Kamenskii-Obukhovskii-Zecca-2001} to conclude that the graph of $U$ is strongly-weakly closed. The latter combined with the compactness of $\gamma$ implies that the graph of $u\mapsto U(\gamma u)$ is weakly-weakly closed. We conclude that the graph of $\Xi$ is closed in $(V\times X^*)_w\times (V\times X^*)_w$.
	
	Therefore, all conditions of Theorem~\ref{fpt} are verified. Using this theorem, we conclude that $\Xi$ has at least a fixed point, say $(u^*,(\eta^*,\xi^*))\in V\times X$.  Hence,  $u^*\in V$ is a weak solution of problem \eqref{eqn1}.
	
	{\bf Weak compactness.}
	The boundedness of the solution set $\Upsilon$ of problem \eqref{eqn1} is a direct consequence of Lemma \ref{lemmaes}. Since $V$ is reflexive, we shall verify  the weak closedness of the solution set $\Upsilon$ of problem \eqref{eqn1}.  Let $u$ be a solution of problem \eqref{eqn1}.  Then, by the definitions of weak solutions (see Definition~\ref{weaksol}) and of $\Xi$, there exist  $(\eta,\xi)\in X^*$ such that $u=S(u,(\eta,\xi))$ and $(\eta,\xi)\in U(\gamma u)$, that is, $(u,(\eta,\xi))\in \Xi(u,(\eta,\xi))$. Let $\{u_n\}_{n\in\N}$ be a sequence of solutions to problem \eqref{eqn1} such that $u_n\wto u$ in $V$ as $n\to\infty$. Hence, we are able to find a sequence  $\{(\eta_n,\xi_n)\}_{n\in\N}\subset X^*$ such that $(u_n,(\eta_n,\xi_n))\in \Xi(u_n,(\eta_n,\xi_n))$. Recall that $U$ is a bounded mapping, it follows that $\{(\eta_n,\xi_n)\}_{n\in\N}$ is bounded in $X^*$. Passing to a subsequence if necessary, we may suppose that $(\eta_n,\xi_n)\wto (\eta,\xi)$ in $X^*$ for some $(\eta,\xi)\in X^*$. Recall that $\Xi$ is weakly sequentially closed,  $(u_n,(\eta_n,\xi_n))\in \Xi(u_n,(\eta_n,\xi_n))$ and $(u_n,(\eta_n,\xi_n))\wto (u,(\eta,\xi))$ in $V\times X^*$, it holds  $(u,(\eta,\xi))\in \Xi(u,(\eta,\xi))$. This means that $u$ is a solution to problem \eqref{eqn1}. Consequently, the solution set of problem \eqref{eqn1} is weakly  compact.
\end{proof}

We end this section by considering special cases of problem \eqref{eqn1}.

If  $J(u)\equiv+\infty$  for all $u\in W^{1,\mathcal H}(\Omega)$, then problem \eqref{eqn1} becomes the non-obstacle mixed boundary value problem (\ref{eqns5}). A careful observation gives the following corollary.

\begin{corollary}\label{corollary}
	Let $p\ge 2$. Assume that \eqref{conditions-p-q-mu}, \textnormal{H($f$)}, \textnormal{H($\phi$)}, \textnormal{H($U_1$)}, \textnormal{H($U_2$)} and \textnormal{H($0$)} are satisfied. Then, the solution set of problem (\ref{eqns5}) is nonempty and weakly compact in $V$.
\end{corollary}

Next, we assume the following conditions.

\begin{enumerate}
	\item[\textnormal{H($j_1$):}]
	$j_1\colon \Omega\times \R\to \R$ and $r_1\colon\R\to \R$ are functions which satisfy the following conditions:
	\begin{enumerate}
		\item[\textnormal{(i)}]
		the function $x\mapsto j_1(x,s)$ is measurable in $\Omega$ for all $s\in\R$ with $x\mapsto j_1(x,0)$ belonging to $L^1(\Omega)$;
		\item[\textnormal{(ii)}]
		for a.\,a.\,$x\in \Omega$, $s\mapsto j_1(x,s)$ is locally Lipschitz continuous and $r_1$ is continuous;
		\item[\textnormal{(iii)}]
		there exist $\theta_3\in[1,p]$, $\alpha_{j_1}\in L^{p'}(\Omega)_+$ and $\beta_{j_1}>0$ such that for all $s\in\R$ and a.\,a.\,$x\in \Omega$, we have
		\begin{equation*}
			|r_1(s)\xi| \le \alpha_{j_1}(x)+\beta_{j_1}|s|^{\theta_3-1}
		\end{equation*}
		for all $\xi\in \partial j_1(x,s)$.
	\end{enumerate}
\end{enumerate}

\begin{enumerate}
	\item[\textnormal{H($j_2$):}]
	$j_2\colon \Gamma_2\times \R\to \R$ and $r_2\colon\R\to \R$ are functions which satisfy the following conditions:
	\begin{enumerate}
		\item[\textnormal{(i)}]
		the function $x\mapsto j_2(x,s)$ is measurable on $\Gamma_2$ for all $s\in\R$ with $x\mapsto j_2(x,0)$ belonging to $L^1(\Gamma_2)$;
		\item[\textnormal{(ii)}]
		for a.\,a.\,$x\in \Gamma_2$, $s\mapsto j_2(x,s)$ is locally Lipschitz continuous and $r_2$ is continuous;
		\item[\textnormal{(ii)}]
		there exist $\theta_4\in[1,p]$, $\alpha_{j_2}\in L^{p'}(\Gamma_2)_+$ and $\beta_{j_2}>0$ such that for all $s\in\R$ and a.\,a.\,$x\in \Gamma_2$, we have
		\begin{equation*}
			|r_2(s)\eta| \le \alpha_{j_2}(x)+\beta_{j_2}|s|^{\theta_4-1}
		\end{equation*}
		for all $\eta\in \partial j_2(x,s)$.
	\end{enumerate}
\end{enumerate}

If $U_1$ and $U_2$ are given by $U_1(x,s)=r_1(s)\partial j_1(x,s)$ for a.\,a.\,$x\in \Omega$, $s\in\R$ and $U_2(x,s)=r_2(s)\partial j_2(x,s)$ for a.\,a.\,$x\in\Gamma_2$, $s\in\R$, then problem \eqref{eqn1} becomes problem \eqref{eqns2}. We have the following result.

\begin{theorem}\label{theorem-special-case}
	Let $p\ge 2$. Assume that \eqref{conditions-p-q-mu}, \textnormal{H($f$)}, \textnormal{H($\phi$)}, \textnormal{H($L$)}, \textnormal{H($J$)}, \textnormal{H($j_1$)}, \textnormal{H($j_2$)} and \textnormal{H($0$)} are satisfied. Then, the solution set of problem \eqref{eqns2} is nonempty and weakly compact in $V$.
\end{theorem}
\begin{proof}
	It is obvious that the conclusion is a direct consequence of  Theorem \ref{maintheorem}. So, we have to verify that the functions $U_1$ and $U_2$, defined by $U_1(x,s)=r_1(s)\partial j_1(x,s)$ for a.\,a.\,$x\in \Omega$, $s\in\R$ and $U_2(x,s)=r_2(s)\partial j_2(x,s)$ for a.\,a.\,$x\in\Gamma_2$, $s\in\R$, fulfill hypotheses \textnormal{H($U_1$)} and \textnormal{H($U_2$)}, respectively.
	
	It follows from Proposition~\ref{P1} that for a.\,a.\,$x\in \Omega$ (resp. for a.a. $x\in\Gamma_2$) and all $s\in\R$ the set $U_1(x,s)$ (resp. $U_2(x,s)$) is nonempty, bounded, closed and convex in $\R$, namely, conditions \textnormal{H($U_1$)} and \textnormal{H($U_2$)} are satisfied. Hypotheses \textnormal{H($j_1$)}(i) and \textnormal{H($j_2$)}(ii) indicate that for all $s\in\R$, the  functions $x\mapsto U_1(x,s)=r_1(s)$ $\partial j_1(x,s)$ and $x\mapsto U_2(x,s)=r_2(s)\partial j_2(x,s)$ are measurable in $\Omega$ and on $\Gamma_2$, respectively. This means that \textnormal{H($U_1$)(i)} and \textnormal{H($U_2$)(i)} hold.
	
	We claim that $s\mapsto r_1(s)\partial j_1(x,s)$ is u.s.c. From Proposition~\ref{uscproposition}, it is sufficient to show that $(r_1(\cdot)\partial j_1(x,\cdot))^-(D)$ is closed for each closed set $D\subset \R$. Let $\{s_n\}_{n\in\N}\subset(r_1(\cdot)\partial j_1(x,\cdot))^-(D) $ be such that $s_n\to s$ as $n\to\infty$. Then, there exists a sequence $\{\eta_n\}_{n\in\N}\subset \R$ satisfying $\eta_n\in r_1(s_n)\partial j_1(x,s_n)\cap D$ for each $n\in\mathbb N$. We are able to find a sequence $\{\xi_n\}_{n\in\N}$ such that $\eta_n=r_1(s_n)\xi_n$ and $\xi_n\in\partial j_1(x,s_n)$ for all $n\in \mathbb N$ and for a.\,a.\,$x\in \Omega$. Recall that $s_n\to s$, we can apply Proposition~\ref{P1}(iii) and (v) to conclude that $\{\xi_n\}_{n\in\N}$ is bounded in $\R$. Hence, without any loss of generality, we may assume that $\xi_n\to \xi$ in $\R$ as $n\to\infty$ for some $\xi\in D$, because of the closedness of $D$. But, the closedness of $\partial j_1$ (see Proposition \ref{P1}(v)) admits that $\xi\in \partial j_1(x,s)$. This combined with the continuity of $r_1$ deduces that $\eta_n=r_1(s_n)\xi_n\to r_1(s)\xi\in r_1(s)\partial j_1(x,s)$. This implies that $s\in (r_1(\cdot)\partial j_1(x,\cdot))^-(D) $, that is, $(r_1(\cdot)\partial j_1(x,\cdot))^-(D) $ is closed. Applying Proposition \ref{uscproposition} we see that $s\mapsto r_1(s)\partial j_1(x,s)$ is u.s.c. Using the same arguments as before, we can also show that $s\mapsto r_2(s)\partial j_2(x,s)$ is u.s.c. Therefore, \textnormal{H($U_1$)}(iii) and \textnormal{H($U_2$)}(iii) are verified.
	
	Finally, hypotheses \textnormal{H($U_1$)(iv)} and \textnormal{H($U_2$)(iv)} are the consequences of the assumptions \textnormal{H($j_1$)}(iii) and \textnormal{H($j_2$)}(iii). Consequently, we apply Theorem \ref{maintheorem} to obtain the desired conclusion.
\end{proof}

\section*{Acknowledgment}

This project has received funding from the NNSF of China Grant Nos. 12001478, 12026255 and 12026256, and the European Union's Horizon 2020 Research and Innovation Programme under the Marie Sklodowska-Curie grant agreement No. 823731 CONMECH, National Science Center of Poland under Preludium Project No. 2017/25/N/ST1/00611,  the Startup Project of Doctor Scientific Research of Yulin Normal University No. G2020ZK07, and the Natural Science Foundation of Guangxi.  The research of Vicen\c tiu D. R\u adulescu was supported by a grant of the Romanian Ministry of Research, Innovation and Digitization, CNCS/CCCDI--UEFISCDI, project number PCE 137/2021, within PNCDI III.

\end{document}